\documentclass[12pt]{amsart}
\usepackage[english]{babel}
\usepackage{a4wide}
\usepackage[T1]{fontenc}
\usepackage{amssymb,amsmath,amsthm,latexsym}
\usepackage{mathrsfs}
\usepackage[usenames,dvipsnames]{color}
\usepackage{graphicx}
\usepackage{mdwlist}
\usepackage{enumerate}
\usepackage{mathtools,dsfont,wasysym}
\usepackage{stmaryrd}
\usepackage{hyperref}
\hypersetup{colorlinks=true,  linkcolor=blue, citecolor=red, urlcolor=cyan}

\newtheorem{theorem}{Theorem}[section]
\newtheorem{lemma}[theorem]{Lemma}

\newtheorem{proposition}[theorem]{Proposition}

\newtheorem{corollary}[theorem]{Corollary}

\theoremstyle{definition}
\newtheorem{definition}[theorem]{Definition}
\newtheorem{example}[theorem]{Example}

\theoremstyle{remark}
\newtheorem{remark}[theorem]{Remark}
\numberwithin{equation}{section}

\newcommand{\R}{\mathbb{R}}
\newcommand{\C}{\mathbb{C}}
\newcommand{\N}{\mathbb{N}}
\newcommand{\K}{\mathbb{K}}

\DeclareMathOperator{\dist}{dist\,}

\DeclareMathOperator{\co}{co}
\DeclareMathOperator{\aco}{aco}

\DeclareMathOperator{\re}{Re}

\DeclareMathOperator{\id}{Id}

\DeclareMathOperator{\Orb}{Orb}

\newcommand{\nn}[1]{{\left\vert\kern-0.25ex\left\vert\kern-0.25ex\left\vert #1 
		\right\vert\kern-0.25ex\right\vert\kern-0.25ex\right\vert}}
\renewcommand{\geq}{\geqslant}
\renewcommand{\leq}{\leqslant}

\newcommand{\norm}[1]{\left\Vert#1\right\Vert}

\newcommand{\NA}{\operatorname{NA}}
\newcommand{\spann}{\operatorname{span}}

\newcommand{\ext}[1]{\operatorname{ext}\left(#1\right)}

\newcommand{\eps}{\varepsilon}

\newcommand{\cconv}{\overline{\co}}

\newcounter{smallromans}

	{\end{list}}

\makeatletter
\renewcommand{\tocsection}[3]{%
	\indentlabel{\@ifnotempty{#2}{\bfseries\ignorespaces#1 #2\quad}}\bfseries#3}
\renewcommand{\tocsubsection}[3]{%
	\indentlabel{\@ifnotempty{#2}{\ignorespaces#1 #2\quad}}#3}

\newcommand\@dotsep{4.5}
\def\@tocline#1#2#3#4#5#6#7{\relax
	\ifnum #1>\c@tocdepth 
	\else
	\par \addpenalty\@secpenalty\addvspace{#2}%
	\begingroup \hyphenpenalty\@M
	\@ifempty{#4}{%
		\@tempdima\csname r@tocindent\number#1\endcsname\relax
	}{%
		\@tempdima#4\relax
	}%
	\parindent\z@ \leftskip#3\relax \advance\leftskip\@tempdima\relax
	\rightskip\@pnumwidth plus1em \parfillskip-\@pnumwidth
	#5\leavevmode\hskip-\@tempdima{#6}\nobreak
	\leaders\hbox{$\m@th\mkern \@dotsep mu\hbox{.}\mkern \@dotsep mu$}\hfill
	\nobreak
	\hbox to\@pnumwidth{\@tocpagenum{\ifnum#1=1\bfseries\fi#7}}\par
	\nobreak
	\endgroup
	\fi}
\AtBeginDocument{%
	\expandafter\renewcommand\csname r@tocindent0\endcsname{0pt}
}
\def\l@subsection{\@tocline{2}{0pt}{2.5pc}{5pc}{}}
\makeatother

\begin{document}
	
\title[Group invariant operators and norm-attaining theory]{Group invariant operators and some applications to norm-attaining theory}
	
\author[Dantas]{Sheldon Dantas}
\address[Dantas]{Departamento de Análisis Matemático, Facultad de Ciencias Matemáticas, Universidad de Valencia, Doctor Moliner 50, 46100 Burjasot (Valencia), Spain\newline
	\href{http://orcid.org/0000-0001-8117-3760}{ORCID: \texttt{0000-0001-8117-3760} } }
\email{\texttt{sheldon.dantas@uv.es}}

\author[Falcó]{Javier Falcó}
\address[Falcó]{Departamento de Análisis Matemático, Facultad de Ciencias Matemáticas, Universidad de Valencia, Doctor Moliner 50, 46100 Burjasot (Valencia), Spain
	\newline
	\href{http://orcid.org/0000-0003-2240-2855}{ORCID: \texttt{0000-0003-2240-2855} }}
\email{\texttt{francisco.j.falco@uv.es}}

\author[Jung]{Mingu Jung$^1$}
\address[Jung]{School of Mathematics, Korea Institute for Advanced Study, 02455 Seoul, Republic of Korea \newline
\href{http://orcid.org/0000-0003-2240-2855}{ORCID: \texttt{0000-0003-2240-2855} }}
\email{\texttt{jmingoo@kias.re.kr}}

\begin{abstract} 
In this paper, we study geometric properties of the set of group invariant continuous linear operators between Banach spaces. In particular, we present group invariant versions of the Hahn-Banach separation theorems and elementary properties of the invariant operators. This allows us to contextualize our main applications in the theory of norm-attaining operators; we establish group invariant versions of the properties $\alpha$ of Schachermayer and $\beta$ of Lindenstrauss, and present relevant results from this theory in this (much wider) setting. In particular, we generalize Bourgain's result, which says that if $X$ has the Radon-Nikodým property, then $X$ has the $G$-Bishop-Phelps property for $G$-invariant operators whenever $G \subseteq \mathcal{L}(X)$ is a compact group of isometries on $X$.
\end{abstract}

\keywords{Bishop-Phelps theorem; Group invariant; Norm-attaining operators; Radon-Nikod\'ym property; separation theorem \\
1: Mingu Jung (jmingoo@kias.re.kr) is the corresponding author. }

\maketitle
	
\tableofcontents

\thispagestyle{plain}

\section{Introduction}

The concept of group invariant mappings has been studied in various forms for over a century, and much longer if we consider the specific case of symmetric mappings, such as symmetric multilinear mappings between Banach spaces, \cite{Dineen, F1, Ryan_thesis}. Despite this, group invariant mappings for a general case in the context of infinite dimensional analysis has not been thoroughly studied until recently (for some results in this direction, we invite the reader to look at the references \cite{AFGM, AFM, AGPZ,F,FGJM}). Our interest here is to study the set of mappings that are invariant under the action of a topological group $G$ of invertible bounded linear operators on a Banach space $X$. The group will be naturally endowed with the relative topology of the Banach space of bounded linear operators.

To be more specific, our focus is on geometric properties of group invariant linear continuous functionals and operators (see Definition \ref{def:G-inv} below). The Hahn-Banach separation theorem and the Bishop-Phelps theorem are, beyond any doubt, two of the most relevant results in the Geometry of Banach Spaces. The first result, which serves as the basis of many geometric and analytic results, provides minimal sufficient conditions for a point to be separated, by a linear continuous functional, from a closed convex set. The second result shows that the set of norm-attaining continuous linear functionals is always norm dense in the dual space, which has been an important tool in optimization theory. 
For deep discussions on these topics we recommend \cite{BP, FHHMZ}.


The purpose of this paper is twofold: on the one hand, we consider group invariant separation theorems motivated by the recent paper \cite{F}, where a version of the Hahn-Banach extension theorem for group invariant functionals was provided (see also \cite{AFM}); on the other hand, we consider the validity of relevant results in the theory of norm-attaining operators with the additional assumption that the involved operators are also group invariant. For background on the topic of norm-attaining operators, we refer the reader to the survey \cite{acosta_survey} and the references there in.

We continue this introduction by presenting the notation and main definitions we use in this work. For us, $X$ and $Y$ will always denote Banach spaces over the set $\K$ of real or complex numbers unless otherwise stated. We denote by $X^*$, $B_X$, and $S_X$ the dual space, the unit ball, and the unit sphere of the Banach space $X$. The set $\mathcal{L}(X, Y)$ denotes the set of all bounded linear operators from $X$ into $Y$. In the particular case that $X=Y$, we simply write $\mathcal{L}(X)$. The letter $G$ is reserved to denote a topological group of invertible bounded linear mappings on $X$ endowed with the relative topology of $\mathcal{L}(X)$.

\begin{definition} \label{def:G-inv} Let $X, Y$ be Banach spaces and let $G\subset \mathcal L(X)$ be a topological group. We say that 

\begin{itemize}
		\setlength\itemsep{0.3em}
	\item[(a)] $T \in \mathcal{L}(X, Y)$ is \emph{$G$-invariant} whenever $T(g(x)) = T(x)$ for every $x \in X$ and $g \in G$,
	   	
	\item[(b)] $K \subseteq X$ is \emph{$G$-invariant} whenever $g(K) \subseteq K$ for every $g \in G$, and 	   	
	\item[(c)] $x \in X$ is \emph{$G$-invariant} whenever the singleton $\{x\}$ is $G$-invariant. 
\end{itemize}
\end{definition} 
By $X_G$ we denote the set of all points $x \in X$ such that $x$ is $G$-invariant. Let us denote by $\mathcal{L}_G (X, Y)$ the set of all $G$-invariant operators from $X$ into $Y$, that is, 
\begin{equation*} 
\mathcal{L}_G(X, Y) := \{T \in \mathcal{L}(X, Y): T \ \mbox{is} \ G \mbox{-invariant} \}.
\end{equation*} 
We define the {\it orbit} of a point $x \in X$ as the set $\Orb(x) = \{g(x): g \in G\}$. It is worth mentioning that the action of a group of permutations $G$ on $\mathbb N$ divides the natural numbers in pairwise disjoint sets (the orbits of the points). We denote this partition of the natural numbers by $\Orb(G,\mathbb N)$.

A classical and elementary example of a space $\mathcal{L}_G (X, Y)$ is the set of symmetric operators from $X=\mathbb R^n$ to $Y=\mathbb R$. Indeed, in this case, 
\begin{equation*}
G = \left \{ \tilde\sigma\in \mathcal L(\mathbb R^n): \sigma\in\Sigma_n \ \ \mbox{and} \ \ \tilde\sigma(x_1,\ldots,x_n)=(x_{\sigma(1)},\ldots,x_{\sigma(n)}) \right \}, 
\end{equation*} 
where $\Sigma_n$ is the group of all permutations on the set of natural numbers $\N = \{1,\ldots,n\}$. Thus, the elements in $\mathcal{L}_G (\mathbb R^n, \mathbb R)$ are operators in $\mathcal{L} (\mathbb R^n, \mathbb R)$ that satisfy $T(x_1,\ldots,x_n)=T(x_{\sigma(1)},\ldots,x_{\sigma(n)})$ for every $(x_1,\ldots,x_n)\in \mathbb R^n$ and the elements in $X_G$ are the points of the form $(x,x,\ldots,x)$ for some $x\in\mathbb R$.

Before we continue, let us set the notation of the norm attaining operators. Recall that $T \in \mathcal{L}(X, Y)$ {\it attains its norm} or it is {\it norm-attaining} if there exists $x_0 \in B_X$ such that $\|T(x_0)\| = \|T\| = \sup_{x \in B_X} \|T(x)\|$. We denote by $\NA(X, Y)$ the set of all operators which are norm-attaining. When $X=Y$, we simply write $\NA(X, X) = \NA(X)$. Moreover, the set of all $G$-invariant norm-attaining operators from $X$ into $Y$ will be denoted by $\NA_G(X, Y)$.

Throughout this paper, all groups are assumed to be compatible with the norm of a given Banach space, that is, given a Banach space $(X, \norm{\cdot})$ and a group $G$ in $\mathcal{L} (X)$, the norm $\norm{\cdot}$ of $X$ is $G$-invariant. In other words, $G$ consists of isometric isomorphisms on $X$. Although it might seem restrictive at first sight, this is a natural and necessary condition when considering norm attaining group invariant operators.

The contents of the paper are as follows. Section \ref{sec:prop} is devoted to present some properties of $G$-invariant mappings. In particular, we consider the behavior of $G$-invariant mappings and $G$-invariant sets with respect to the operation of taking limits and symmetrising. These properties will be used in Sections \ref{sec:sep} and \ref{sec:NA}, where we present the main results of the manuscript. More specifically, in Section \ref{sec:sep}, we consider separation results related to the Hahn-Banach separation theorems; in Section \ref{sec:NA}, we extend some results from the theory of norm-attaining operators to the context of $G$-invariant mappings: we start by considering some of the classical counterexamples in the theory in order to see whether these counterexamples  in the norm attaining theory remain valid or not depending on the specific choice of the group $G$. After that, we provide the analogues to properties $\alpha$ of Schachermayer \cite{S} and $\beta$ of Lindenstrauss \cite{L}, as well as their generalizations; properties quasi-$\alpha$ and quasi-$\beta$ (see \cite{CS}). To conclude, we define and study the $G$-Bishop-Phelps property, which is a $G$-invariant version of the Bishop-Phelps property introduced by Bourgain \cite{B}, and show that the Radon-Nikod\'ym property is sufficient to ensure the $G$-Bishop-Phelps property, provided that the group $G$ is compact.

\section{Some properties of $G$-invariant mappings}\label{sec:prop}

In this section, we provide straightforward but essential basic facts on group invariant operators. They will be the key to prove the results in the following sections. We start with the following property whose proof can be obtained directly from the definition of the $G$-invariance of an operator (see Definition \ref{def:G-inv} above).

\begin{proposition} \label{lemma-convergence} Let $X, Y$ be Banach spaces. Let $G \subseteq \mathcal{L}(X)$ be a group of isometries. Suppose that $(T_n)_{n=1}^{\infty} \subseteq \mathcal{L}(X, Y)$ is a sequence which converges (in the strong operator topology) to $T_{\infty}$. If $T_n$ is $G$-invariant for every $n \in \N$, then so is $T_{\infty}$. 
\end{proposition}

Let $G \subseteq \mathcal{L}(X)$ be a compact group of isometries. Given $x \in X$, the {\it $G$-symmetrization} of $x$ is the point $\overline{x} \in X$ given by 
\begin{equation*} 
\overline{x} := \int_G g(x) \, d\mu(g), 
\end{equation*} 
where $\mu$ is the normalized Haar measure on the group $G$, for more details see \cite{AGPZ} and \cite[Chapter 7]{ CLO}. Notice that if $x \in X$ is already $G$-invariant, then $\overline{x} = x$. This, in particular, implies that the mapping $x \mapsto \overline{x}$ is a norm-one projection on $X$ with range $X_G$, that is, $X_G$ is a one-complemented subspace of $X$. When it comes to closed convex sets, we have the following result. We denote by $\co(A)$ and $\overline{\co}(A)$ the convex hull and the closed convex hull for a set $A$, respectively.

\begin{proposition} \label{lemma1} Let $X$ be a Banach space and $C$ be a closed convex subset of $X$. Let $G$ be a compact group of $\mathcal{L}(X)$ such that $C$ is $G$-invariant. Then, for $z \in C$, its symmetrization $\overline{z}$ belongs to $C$.
\end{proposition}

\begin{proof} Let $z \in C$ be given. Notice that 
	\begin{equation*}
		\overline{z} = \int_G f_z(g) d \mu(g)
	\end{equation*}
	for $f_z: G \rightarrow X$ is given by $f_z(g) = g(z)$, where we are considering the Bochner integral and $\mu$ is the normalized Haar measure on $G$. For $\delta > 0$, pick  $g_1, \ldots, g_N$ of $G$ and set $U_i:= \{ r \in G : \| r - g_i\| < \delta\}$ for each $i =1,\ldots, N$, that is a $\delta$-covering of $G$. Define $E_1 := U_1$ and $E_i := U_i \setminus \bigcup_{j=1}^{i-1} U_j$ for $i = 2, \ldots, N$. Notice that if we define $h: G \rightarrow X$ by 
	\begin{equation*}
		h(g) := \sum_{i=1}^N g_i(z) \chi_{E_i}(g),
	\end{equation*} 	
	then $\int_G h(g) d \mu(g) \rightarrow \int_G f_z(g) d \mu(g)$ as $\delta \rightarrow 0$. As
	\begin{equation*}
		\int_G h(g) d \mu(g) = \sum_{i=1}^N g_i(z) \mu(E_i) \in \co(C) = C, 
	\end{equation*} 	
	we have that $\overline{z} \in \overline{C} = C$.
\end{proof}

Given a compact group $G \subseteq \mathcal{L}(X)$ of isometries and $T \in \mathcal{L}(X,Y)$, the {\it $G$-symmetrization} of $T \in \mathcal{L} (X, Y)$ is given by 
\begin{equation} \label{symmetrization} 
\overline{T} (x):= \int_G T(g(x)) \, d\mu(g),
\end{equation} 
where $\mu$ is the normalized Haar measure on $G$. It is clear that $\overline{T}$ is $G$-invariant and the inequality $\|\overline{T}\| \leq \|T\|$ holds true in general. This implies that, as for the case of $G$-invariant points, the mapping $T \mapsto \overline{T}$ is a norm-one linear projection with range $\mathcal{L}_G (X, Y)$ and, therefore, $\mathcal{L}_G (X, Y)$ is a one-complemented subspace of $\mathcal{L} (X, Y)$. Furthermore, it turns out that $B_{X_G}$ is norming for the space $\mathcal{L}_G(X, Y)$ according to the next result.


\begin{proposition} \label{norming} Let $X, Y$ be Banach spaces. Let $G \subseteq \mathcal{L}(X)$ be a compact group of isometries. Then, $B_{X_G}$ is norming for $\{ T \in \mathcal{L} (X,Y): \|T \| = \| \overline{T} \|\}$. In particular, $B_{X_G}$ is norming for $\mathcal{L}_G (X, Y)$. 
\end{proposition}

\begin{proof} Let $T \in \mathcal{L} (X, Y)$ be such that $\|T \| = \|\overline{T} \|$ and $\eps >0$ be given. Pick $x_0 \in B_X$ to be such that $\|\overline{T} (x_0)\| > \|\overline{T}\| - \eps$. By Proposition \ref{lemma1} (or simply by definition), the symmetrization $\overline{x}_0$ of $x_0$ satisfies that $\|\overline{x}_0\| \leq \|x_0\|$. Observe that (we send the reader to equation (\ref{symmetrization}) and the comments around it for more information about $\overline{T}(\overline{x_0})$ that we use in what follows)
	\begin{align*}
		\overline{T}(\overline{x}_0) = \int_G (\overline{T} \circ g)(x_0) \, d \mu(g) = \overline{T} (x_0) \,\text{ and }\, \overline{T}(\overline{x}_0) = \int_G T (g(\overline{x}_0)) \, d\mu = T(\overline{x}_0). 
	\end{align*}
	Therefore,
	\[
		\|T(\overline{x}_0)\| = \left\| \overline{T} (\overline{x}_0) \right\| = \|\overline{T} (x_0) \| > \|T\| -\eps. 
	\]
\end{proof}

We conclude this section by discussing some basic facts on the adjoint of linear operators. These facts will be used in Subsection \ref{sectionadj}. 
Notice that whenever $G \subseteq \mathcal{L}(X)$ is a group, so is $G^{**} := \{ g^{**} \in \mathcal{L} (X^{**}) : g \in G\}$. In fact, we have the following characterization that shows, in particular, that an operator is $G$-invariant if and only if its second adjoint is $G^{**}$-invariant.

\begin{proposition} \label{lemma-caract-dual-bidual} Let $X, Y$ be Banach spaces. Let $G \subseteq \mathcal{L}(X)$ be a group of isometries and $T \in \mathcal{L}(X, Y)$. The following statements are equivalent.
	\begin{itemize}
			\setlength\itemsep{0.3em}
		\item[(a)] $T$ is $G$-invariant,
		\item[(b)] $T^*y^*$ is $G$-invariant for every $y^* \in Y^*$, and
		\item[(c)] $T^{**}$ is $G^{**}$-invariant. 
	\end{itemize}
\end{proposition}

\begin{proof} {\it (a) $\Rightarrow$ (b)}. Suppose that $T$ is $G$-invariant. Let $y^* \in Y^*$ be given. For every $x \in X$ and $g \in G$, we have that $(T^*y^*)(g(x)) = y^*(T(g(x))) = y^*(T(x)) = (T^* y^*)(x)$.

	\noindent 
	{\it (a) $\Rightarrow$ (c)}. Suppose that $T$ is $G$-invariant. Let $x \in X$ and $g^{**} \in G^{**}$. Then, for every $y^* \in Y^*$, we have $[(T^{**} \circ g^{**} )(x)](y^*) = g(x)(T^*(y^*)) = T^*(y^*)(g(x)) = [(T \circ g)(x)](y^*) = (T(x))(y^*)$. This shows that $(T^{**} \circ g^{**} ) (x) = T(x)$ for every $x \in X$ and, consequently, that $T^{**} \circ g^{**} = T^{**}$ on $X$. By the $w^*$-continuity of $T^{**}$, we have the $T^{**} \circ g^{**}  = T^{**}$ on $X^{**}$ for every $g^{**} \in G^{**}$.

	\noindent
	{\it (b) $\Rightarrow$ (a)}. Indeed, for every $x \in X, y^* \in Y^*$, and $g \in G$, we have that $y^*(T(g(x))) = (T^*y^*)(g(x)) = (T^*y^*)(x) = y^*(T(x))$. By the Hahn-Banach theorem, $T(g(x)) = T(x)$ for every $g \in G$ and every $x \in X$.

	\noindent
	{\it (c) $\Rightarrow$ (a)}. Let $x \in X, y^* \in Y^*$, and $g \in G$. We have that $y^*(T(g(x))) = (T^* y^*)(g(x)) = T^{**}(g(x))(y^*) = T^{**}(x)(y^*) = y^*(T(x))$.	This implies, again by the Hahn-Banach theorem, that $T(g(x)) = T(x)$ for every $g \in G$ and every $x \in X$.	
\end{proof}

\section{Group invariant separation theorems}\label{sec:sep}

In \cite{F}, a version of the Hahn-Banach extension theorem for group invariant mappings was considered and it played an important role in the proof of a group invariant Bishop-Phelps theorem (see \cite[Theorem 3]{F}). With the aid of this extension theorem, we can obtain the following versions of the Hahn-Banach separation theorems for group invariant mappings under the assumption, as in \cite[Proposition 1]{F}, that the given group is {\it compact}. 

\begin{theorem}[Hahn-Banach separation theorem] \label{HB-separation} Let $X$ be a Banach space and $G \subseteq \mathcal{L}(X)$ be a compact group of isometries. If $C$ is a closed convex $G$-invariant subset of $X$ and $x_0$ is a $G$-invariant point such that $x_0 \not\in C$, then there is $x^* \in X^*$ which is $G$-invariant such that
	\begin{equation*} 
		\re x^* (x_0) > \sup \{ \re x^* (x): x \in C\}.
	\end{equation*} 
\end{theorem}

\begin{proof} Let us first assume that $X$ is a real space. Given $x \in C$, by Proposition \ref{lemma1}, its symmetrization $\overline{x}$ belongs to $C$. By considering the set $C - \overline{x}$ (which is again a $G$-invariant set) and the point $x_0 - \overline{x}$, we may assume that $0 \in C$. Let
	\begin{equation*}
		\delta:= \dist(x_0, C) > 0 \ \ \ \ \ \mbox{and} \ \ \ \ \ D:= \left\{ x \in X: \dist(x, C) \leq \frac{\delta}{2} \right\}.
	\end{equation*}
	We claim that $D$ is $G$-invariant. 
	Assume to the contrary that it is not. In this case, there are $x \in D$ and $g \in G$ such that $\dist(g(x), C) > r > \frac{\delta}{2}$. Since $\norm{\cdot}$ is $G$-invariant, we then have that $\|x - g^{-1}(y)\| = \|g(x) - y\| > r$ for every $y \in C$.
	As $C$ is $G$-invariant, we get that $\dist(x, C) \geq r > \frac{\delta}{2}$, which yields a contradiction.
	
	\vspace{0.2cm}
		
	Now, let $\mu_D$ be the Minkowski functional of $D$, that is, $\mu_D(x) = \inf \{ \lambda > 0: x \in \lambda D\}$. Notice that if $x \in \lambda D$, then $g(x) \in \lambda D$ since $D$ is $G$-invariant. This implies that $\mu_D(g(x)) \leq \mu_D(x)$. Symmetrically, we have that $\mu_D(x) \leq \mu_D(g(x))$. Hence, $\mu_D$ is a $G$-invariant functional. Moreover, notice that since $D$ is closed and $x_0 \not\in D$, we have that $\mu_D(x_0) > 1$. Define a linear functional $y^*$ on $\spann \{x_0\}$ by $y^* (\lambda x_0) := \lambda \mu_D(x_0)$. Notice that $\spann \{x_0\}$ is $G$-invariant and $y^* $ is also $G$-invariant. Furthermore, $y^* (\lambda x_0) \leq \mu_D(\lambda x_0)$. By \cite[Proposition 1]{F}, there exists a $G$-invariant continuous linear functional $x^*$ on $X$ such that $x^*=y^*$ on $\spann\{x_0\}$ and $x^* (x) \leq \mu_D(x)$ for every $x \in X$. Finally, notice that that 
	\begin{equation*}
		x^*(x_0) = y^* (x_0) = \mu_D(x_0) > 1 \geq \mu_D(x) \geq x^* (x)
	\end{equation*}
	for every $x \in C$. 
	
	If $X$ is a complex space, then we construct $g \in X_{\mathbb{R}}^*$ as in the real case (where ${X_{\mathbb{R}}}$ denotes the real part of $X$) and define $f \in X^*$ by $f(x)=g(x) - i g(ix)$ for every $x \in X$. Then $f$ satisfies the desired result.
\end{proof} 

It is worth remarking that the above proof is based on the standard proof of the classical Hahn-Banach separation theorem (see, for instance, \cite[Theorem 2.12]{FHHMZ}). We present a shorter alternative proof for this result by making use of symmetrizations.

\begin{proof}[Alternative proof for Theorem \ref{HB-separation}]
Due to the classical Hahn-Banach separation theorem, there exists $x_0^* \in X^*$ such that 
\[
\re x_0^* (x_0) > \sup \{ \re x_0^* (x): x \in C\}.
\]
If we denote by $x^*$ the $G$-symmetrization of $x_0^*$, then 
\[
\re x^* (x) = \int_G \re x_0^* (g(x)) \, d\mu (g) \leq \sup \{ \re x_0^* (x) : x \in C \}
\]
for every $x \in C$ (here we are using that $C$ is $G$-invariant). In other words, we obtain that $\sup \{ \re x^* (x) : x \in C \} \leq \sup \{ \re x_0^* (x) : x \in C \}$. Therefore, it follows from the fact that $x_0$ is $G$-invariant that
\begin{align*}
\re x^* (x_0) = \int_G \re x_0^* (g(x_0)) \, d\mu(g) = \re x_0^* (x_0) > \sup \{ \re x^* (x) : x \in C \} 
\end{align*} 
as desired.
\end{proof} 


\begin{remark} Let us notice that we {\it do need} to assume that $x_0$ is $G$-invariant in Theorem \ref{HB-separation} above. Indeed, let $X = (\R^2,\Vert\cdot\Vert_\infty)$ and $G := \{ \id_X, \sigma_{1,2}\}$, where $\sigma_{1,2}(x, y) = (y, x)$. Then, $G$ is a compact group and $B_{X}$ is a closed convex $G$-invariant subset of $X$. Notice that, for every functional $x^*: X \rightarrow \R$ that is $G$-invariant, we have that
	\begin{equation*}
		x^*(x, -x) = x^* (\sigma_{1,2}(x, -x)) = x^*(-x, x)
	\end{equation*} 
	for every $x \in \R$. This implies that $x^*(x, -x) = 0$ for every $x \in \R$. Therefore, the point $(x, -x) \notin B_{X}$ cannot be separated from $B_{X}$ by a $G$-invariant functional on $X$. 
\end{remark}

For more examples of compact groups $G$, $G$-invariant closed convex sets, $G$-invariant operators and $G$-invariant points we recommend \cite[section 4]{AFM} and \cite{AGPZ}.

We also have the following result which will be used to prove Lemma \ref{lemma-bourgain}.

\begin{proposition} \label{functional} Let $X$ be a Banach space and $G \subseteq \mathcal{L}(X)$ be a compact group of isometries. Suppose that $C$ is a convex $G$-invariant subset of $X$ and $x_0$ is a $G$-invariant point such that $\dist(x_0, C) > \delta > 0$. Then, there exists a functional $x^* \in S_{X^*}$ which is $G$-invariant such that
	\begin{equation*}
		\re x^* (x_0) > \sup \{ \re x^* (x): x \in C \} + \delta.
	\end{equation*}
\end{proposition}

\begin{proof} Since $\dist(x_0, C) > \delta > 0$, we have that $\{x_0\} \cap \overline{(C + \delta B_X)} = \emptyset$. Notice now that $\overline{C + \delta B_X}$ is $G$-invariant, closed, and convex. By Theorem \ref{HB-separation}, there exists $y^* \in X^*$ which is $G$-invariant such that
	\begin{eqnarray*}
		\re y^* (x_0) &>& \sup \left\{ \re y^* (x) + \delta y^* (z): x \in C, z \in B_X \right\} \\
		&=& \sup \{ \re y^* (x): x \in C\} + \delta \|y^* \|. 
	\end{eqnarray*}	
In particular $y^* \not=0$. Setting $x^*:= \frac{y^*}{\|y^*\|} \in S_{X^*}$, we get that 
	\begin{equation*}
		\re x^*(x_0) > \sup \{ \re x^*(x): x \in C\} + \delta
	\end{equation*}	
	as we wanted.
\end{proof}

From \cite[Proposition 1]{F}, we have that if $G \subseteq \mathcal{L} (X)$ is a compact group and $x \in S_X$ is a $G$-invariant point, then there exists a $G$-invariant functional $x^* \in S_{X^*}$ such that $x^* (x) = 1$. Our next result, which is of independent interest, shows that the existence of a $G$-invariant supporting functional of a (not necessarily $G$-invariant) point $x \in S_X$ is closely related to the convex hull of the orbit of $x$.

\begin{proposition} \label{chain}
	Let $X$ be a Banach space. Let $G \subseteq \mathcal{L} (X)$ be a group of isometries and $x \in S_X$. 
	\begin{enumerate}
		\item[(i)] If there exists a $G$-invariant supporting functional $x^*$ of $B_X$ at $x$, then $\co \{ g(x) : g \in G \} \subseteq S_X$.
		\item[(ii)] If there is an ascending chain $\mathcal{C}$ of compact subgroups of $G$ such that $\cup_{S \in \mathcal{C}} \,S$ is dense in $G$ (in the strong operator topology) and $\co \{ g(x) : g \in G\} \subseteq S_X$, then there exists a $G$-invariant supporting functional $x^*$ of $B_X$ at $x$. 
	\end{enumerate} 
\end{proposition} 

\begin{proof} For (i), notice that
\begin{equation*} 	
1 \geq \Big\| \sum_{i=1}^n \lambda_i g_i (x)\Big\| \geq x^* \Big( \sum_{i=1}^n \lambda_i g_i (x) \Big) =   \sum_{i=1}^n \lambda_i x^*\big(g_i (x) \big)=   \sum_{i=1}^n \lambda_i x^* (x) =   \sum_{i=1}^n \lambda_i= 1 
\end{equation*} 
for $\lambda_i \geq 0$ and $g_i \in G$, $i=1,\ldots, n$, with $\sum_{i=1}^n \lambda_i = 1$. Now let us prove (ii). In order to do so, we use arguments similar to \cite[Theorem 2.5]{AGPZ}. Let $S \in \mathcal{C}$ be fixed. Let us denote by $\sigma_S (x)$ the $S$-symmetrization of $x$. From the argument from the proof of Proposition \ref{lemma1}, we have that $\sigma_S (x) \in \cconv \{ g(x) : g \in G \} \subseteq S_X$. By the Hahn-Banach theorem \cite[Proposition 1]{F}, we can find a $S$-invariant functional $x_S^* \in S_{X^*}$ such that $x_S^* (\sigma_S (x)) = 1$. On the other hand, 
	\[
		x_S^* (\sigma_S (x)) = \int_S x_S^* (g(x)) \, d\mu_S (g) = \int_S x_S^* (x) \, d\mu_S (g) = x_S^* (x),
	\]
	where $\mu_S$ denotes the normalized Haar measure on $S$. Thus, $x_S^* (x) = 1$. Now, take a free ultrafilter $\mathcal{U}$ on the index set of compact subgroups of $G$. Let us define the mapping $x^*$ on $X$ by 
	\[
		x^* (u) = \lim_{\mathcal{U}} x_S^* (u)
	\] 
	for each $u \in X$. The limit is well-defined since $\{ x_S^* (u): S \in \mathcal{C} \}$ is a bounded set. Note that $x^*$ is a continuous linear functional that satisfies $\|x^*\| = x^* (x) = 1$. Moreover, $x^* (h(x)) = x^* (x)$ for every $h \in \cup_{S \in \mathcal{C}} S$ and every $x \in X$. Thus, the denseness assumption assures that $x^*$ is $G$-invariant and this completes the proof. 
\end{proof} 

We have the following consequence of Proposition \ref{chain}. 

\begin{corollary} \label{cor:chain}
	Let $X$ be a Banach space and $G \subseteq \mathcal{L} (X)$ be a compact group of isometries. Then, every point $x \in S_X$ is a supporting point of a $G$-invariant supporting functional $x^*$ of $B_X$ at $x$ if and only if $\co \{ g(x) : g \in G \} \subseteq S_X$ for every $x \in S_X$. 
\end{corollary}

As a particular case of Corollary \ref{cor:chain}, we have that if $G$ is the trivial group, then we recover the well-known result that every point of the unit sphere $S_X$ of a Banach space $X$ is a supporting point of the closed unit ball $B_X$ of $X$.

\section{Norm-attaining $G$-invariant operators}\label{sec:NA}

Our aim in this section is to extend some relevant results from the theory of norm-attaining operators under the hypothesis that the operators that we are considering are invariant under a group action. This approach does not depend only on the underlying Banach spaces but also on the group acting on $X$. Indeed, we start by studying some counterexamples to the Bishop-Phelps theorem in the scenario of $G$-invariant operators. Many of these counterexamples (see \cite{acosta1, G, JW, L, S}) consider $X$ to be a Banach sequence space, that is, a Banach space whose elements are infinite sequences of real or complex numbers. In this scenario, the most natural group to consider is a group of permutations of the natural numbers. To simplify the notation, we will denote in the same way the subgroup $G$ of permutations of the natural numbers and the natural group associated to the sequence space $X$ whose action is naturally defined by 
\begin{equation*} 
g(x_1,x_2,x_3,\ldots)=(x_{g(1)},x_{g(2)},x_{g(3)},\ldots) 
\end{equation*} 
for every $g\in G$ and every $(x_1,x_2,x_3,\ldots)\in X$.

The first counterexample that we consider is based on Lindenstrauss' work \cite{L}, where he shows that if $Y$ is a strictly convex renorming of $c_0$, then the set $\NA(c_0, Y)$ is {\it not} dense in $\mathcal{L}(c_0, Y)$. Let us show that, for invariant operators, the validity of this result depends on the properties of the group itself. In fact, the key point here is the structure of the orbits of the action of the subgroup of permutations $G$ on the natural numbers. As we have mentioned above, the action of a group of permutations $G$ on $\mathbb N$ divides the natural numbers in pairwise disjoint sets, which we denote by $\Orb(G,\mathbb N)$. Our interest is the set of orbits that are finite. For Propositions \ref{example1}, \ref{d-star}, and \ref{example3}, we will be using the following notation for this set:
\begin{equation*} 
\mathcal F_{G,\mathbb N}:=\{H\in \Orb(G,\mathbb N): |H|<+\infty\},
\end{equation*} 
where $|H|$ denotes the cardinality of $H$. For simplicity, we use the notation $\Orb(n)$ for the set $\{ g(n) : g \in G\}$.

\begin{proposition} \label{example1} 
	Let $G$ be a group of permutations of the natural numbers and $Y$ a Banach space. 
	\begin{enumerate}
			\setlength\itemsep{0.3em}
	\item
	If $\mathcal F_{G,\mathbb N}$ is finite, then every $G$-invariant operator attains its norm, that is, 
\begin{equation*} 
\NA_G (c_0, Y) = \mathcal{L}_G(c_0, Y).
\end{equation*} 
\item If $\mathcal{F}_{G, \mathbb{N}}$ is infinite and $Y$ is a strictly convex renorming of $c_0$, then 
\begin{equation*} 
\overline{\NA_G (c_0, Y)} \neq \mathcal{L}_G(c_0, Y). 
\end{equation*} 
\end{enumerate} 
\end{proposition}

\begin{proof} Set $(e_i)_{i=1}^{\infty}$ to be the canonical basis of $c_0$. 
First, suppose that $\mathcal F_{G,\mathbb N}$ is finite and $Y$ is a Banach space. Without loss of generality and to simplify our notation, we can assume that there exists $N \in \N$ such that 
	\[
		\mathcal F_{G,\mathbb N}= \Big\{ H_1=\{1,\ldots,{n_1}\},H_2=\{{n_1+1},\ldots,{n_2}\},\ldots,H_N=\{{n_{N-1}+1},\ldots,{n_N}\} \Big\}.
	\]
Fix an arbitrary element $T \in \mathcal{L}_G(c_0, Y)$ and let $n > n_N$ (thus, $\Orb(n)$ is not finite). Then, for $k_1,\ldots,k_m\in \Orb(n)$ with $m \in \mathbb{N}$, we have that $\Vert \sum_{i=1}^m e_{k_i}\Vert=1$ and since $T$ is $G$-invariant and linear we have that $T\left(\sum_{i=1}^m e_{k_i}\right)=mT(e_n)$. Since $m$ is arbitrary and the operator $T$ is bounded, we have that $T(e_n)=0$. Thus, the operator $T$ only depends on the first $n_N$ coordinates, that is, 
	\begin{equation*}
		T(x) = T((x_1, \ldots, x_{n_N}, 0, 0, 0, \ldots ))
	\end{equation*}
	for every $x= (x_n)_{n=1}^{\infty} \in c_0$. It follows that $T \in \NA_G (c_0, Y)$.
	
	Next, we assume that the set $\mathcal F_{G,\mathbb N}$ is infinite and $Y$ is a strictly convex renorming of $c_0$. Consider an enumeration $H_1, H_2, \ldots$ of the sets in $\mathcal F_{G,\mathbb N}$. We define the linear operator $T$ from $c_0$ into $Y$ given by 
	\[
		T(e_n) :=\begin{cases}
		e_j, &	\text{ if }n\in H_j,\\
		0, &\text {otherwise}.
		\end{cases}
	\]
The operator $T$ is clearly linear, bounded, and $G$-invariant. However, any norm-attaining operator $R$ from $c_0$ into a strictly convex Banach space satisfies that $R(e_n)=0$ for every $n$ bigger than some $n_R \in \mathbb{N}$ (see, for instance, \cite[Lemma 2]{Miguel}). Therefore, we have that the operator $T$ cannot be approximated by norm-attaining operators.
\end{proof}

In the context of bounded linear operators, more examples where the Bishop-Phelps theorem fails have been considered after Lindenstrauss' work. For instance, W. Gowers \cite{G} proved that $\NA(d_{*}(w, 1), \ell_p)$ is {\it not} dense in $\mathcal{L}(d_{*}(w, 1), \ell_p)$ for $1 < p < \infty$, where $d_{*}(w, 1)$ is the predual of the Lorentz sequence space provided that $w = (1/n)_{n=1}^{\infty}$. In Proposition \ref{d-star}, we give a full characterization of when the equality $\overline{\NA_G (d_{*}(w, 1), \ell_p)} = \mathcal{L}_G(d_{*}(w, 1), \ell_p)$ holds true in the case that $G$ is a group of permutations of $\N$ such that the size of orbits is bounded. Before doing so, let us recall the definitions of  the Lorentz sequence space $d(w, 1)$ and its predual $d_{*}(w, 1)$: 
\[
	d(w,1) = \Big\{ x= (x_i)_{i=1}^{\infty} \in c_0 : \sup \Big\{ \sum_{n=1}^\infty |x(\sigma(n))|w_n : \sigma : \mathbb{N} \rightarrow \mathbb{N} \text{ injective }\Big\} < \infty \Big\} 
\]
endowed with the norm 
\[
	\| x \| = \sup \Big\{ \sum_{n=1}^\infty |x(\sigma(n))|w_n : \sigma : \mathbb{N} \rightarrow \mathbb{N} \text{ injective }\Big\}
\]
and 
\[
	d_{*}(w,1)=\left\{x=(x_i)_{i=1}^{\infty} \in c_0: \lim_{k\to\infty}\frac{\sum_{i=1}^k[x_i]}{\sum_{i=1}^k w_i} \right\}
\] 
where $([x_i])_{i=1}^{\infty}$ is the non-increasing rearrangement of $(|x|_i)_{i=1}^{\infty}$ endowed with the norm 
\[
	\Vert x\Vert=\sup_k \frac{\sum_{i=1}^k[x_i]}{\sum_{i=1}^k w_i}.
\]

We recall the well-known fact that the usual vector basis $(e_n)$ is a monotone Schauder basis of $d(w,1)$ and $d_* (w,1)$ has a Schauder basis whose sequence of biorthogonal functionals is the basis of $d(w,1)$ (see \cite{Gar, Wer}). 

\begin{proposition} \label{d-star} 
	Let $G$ be a group of permutations of the natural numbers. Consider also $w = (1/n)_{n=1}^\infty \in c_0$ and assume that $1 < p < \infty$.
\begin{enumerate}
		\setlength\itemsep{0.3em}
		\item If $\mathcal F_{G,\mathbb N}$ is finite, then, for any Banach space $Y$, every $G$-invariant operator from $d_{*}(w,1)$ into $Y$ attains its norm. In other words, 
\begin{equation*} 
\NA_G (d_{*}(w, 1), Y) = \mathcal{L}_G(d_{*}(w, 1), Y).
\end{equation*} 	
\item If $\mathcal F_{G,\mathbb N}$ is infinite and there exists $C >0$ such that $ |H| < C$ for every $H \in \mathcal F_{G,\mathbb N}$, then 
$\overline{\NA_G (d_{*}(w, 1), \ell_p)} \neq \mathcal{L}_G(d_{*}(w, 1), \ell_p)$.
	\end{enumerate} 
\end{proposition}

Before proving Proposition \ref{d-star}, let us point out that if there is an infinite partition $\{ I_n \}$ of $\mathbb{N}$ such that $\sup_n |I_n| < C$ for some $C>0$, then any group $G$ naturally induced by permutation groups on the sets $I_n$ will satisfy the hypothesis in condition (2) above.  

\begin{proof}[Proof of Proposition \ref{d-star}] Assume first that $\mathcal F_{G,\mathbb N}$ is finite. As in the proof of Proposition \ref{example1}, we may assume that there exists $N \in \N$ such that 
\[
		\mathcal F_{G,\mathbb N}= \Big\{ H_1=\{1,\ldots,{n_1}\},H_2=\{{n_1+1},\ldots,{n_2}\},\ldots,H_N=\{{n_{N-1}+1},\ldots,{n_N}\} \Big\}.
	\]
	 Let $Y$ be an arbitrary Banach space and $T \in \mathcal{L}_G(d_{*}(w, 1), Y)$ be given. Fix $n_0 > n_N$ (so, $\Orb(n_0)$ is not finite). Denote by $(e_n)$ the standard vector basis of $d_* (w,1)$. We claim that $\|T (e_{n_0})\| = 0$. To this end, write $\Orb (n_0) = \{m_1, m_2, \ldots \}$ and observe that for $M \in \N$
	 \[
	y_M:= \sum_{j=1}^{M} c_M e_{m_j} \in S_{d_* (w,1)}, \,\, \text{ where} \,\, c_M = \frac{\sum_{i=1}^M \frac{1}{i}}{M}. 
	 \]
Since $T$ is $G$-invariant, we have 
\[
\|T\| \geq \|T(y_M)\| = \Big\| \sum_{j=1}^{M} c_M T(e_{m_j}) \Big\| = M c_M \|T(e_{n_0})\| = \Big(\sum_{i=1}^M \frac{1}{i} \Big) \|T(e_{n_0})\|. 
\] 
Note that the right-hand side tends to infinity as $M \rightarrow \infty$ if $T(e_{n_0}) \not= 0$. Thus, $\|T (e_{n_0})\| = 0$ and this shows that $T$ depends only on finitely many coordinates. Therefore, $T \in \NA_G(d_{*}(w, 1), Y)$ as we wanted to prove.

	Assume now that $\mathcal F_{G,\mathbb N}$ is not finite and consider an enumeration $H_1, H_2,\ldots$ of the sets in $\mathcal F_{G,\mathbb N}$. By assumption, there is $C >0$ such that $|H_n| < C$ for all $n \in \N$. For each $n \in \N$, consider the finite dimensional subspace of $d_* (w,1)$ given by $X_n :=\spann\{e_k:k\in H_n \}$. Let $f_n$ be a norm-one functional on $X_n$ that it is invariant under the action of the group $H_n$ (considered as a group in $\mathcal{L} (X_n)$). 
	Consider the operator $T \in \mathcal{L} (d_{*} (w,1), \ell_p)$ given by 
	\[
		T(e_j)= (0,\ldots, 0, \underbrace{f_n (e_j)}_{n ^{\text{th}}-\text{term}}, 0, 0, \ldots )
	\]
 if $j\in H_n$ for some $n \in \N$ and $T(e_j) = 0$, otherwise. 
	Then, by using Jensen's inequality, 
	\begin{align*}
		\Vert T(x)\Vert^p = \sum_{n =1}^\infty \left| \sum_{j \in H_n } f_n (e_j) x_j \right|^p \leq \sum_{n=1}^\infty \sum_{j \in H_n} |H_n|^{p-1} |f_n (e_j)|^p |x_j|^p \leq 
		C^{p-1} \sum_{n=1}^\infty |x_n|^p 
	\end{align*} 
	for each $x=(x_1,x_2,\ldots) \in d_{*}(w, 1)$. Recall that the formal identity mapping from $d_{*}(w, 1)$ into $\ell_p$ is bounded and every operator from $d_{*}(w, 1)$ into $\ell_p$ that attains its norm has finite support (see the proof of \cite[Theorem, p. 149]{G}). This implies that $T$ is a $G$-invariant bounded linear operator from $d_{*}(w, 1)$ into $\ell_p$ which cannot be approximated by norm-attaining operators. 
	As a matter of fact, if $S$ is a norm-attaining operator from $d_{*}(w, 1)$ into $\ell_p$, then there exists $N \in \N$ so that $S(e_k)=0$ for $k>N$. For any $n \in \N$ with $\min H_n > N$, we have that 
\begin{equation*} 
	\| T - S \| \geq \sup \{ \| (T-S) (x) \| : x \in B_{X_n} \} = \sup \{ | f_n (x) | : x \in B_{X_n} \} = \Vert f_n\Vert_{X_n}=1 
\end{equation*} 
and this completes the proof.
\end{proof}

Let $w:=(w_n)_{n=1}^{\infty}$ be a non-increasing sequence of positive numbers such that $w \not\in \ell_1$ and $w_1 < 1$. Consider the Banach space $Z$ of the sequences $z$ of scalars endowed with the norm given by 
\begin{equation*} 
\|z\| := \|(1 - w)z\|_{\infty} + \|wz\|_1,
\end{equation*} 
where $\|\cdot\|_{\infty}$ and $\|\cdot\|_1$ are the usual $\ell_{\infty}$- and $\ell_1$-norms, respectively, and $zw$ denotes the point-wise multiplication of the sequences. Denote by $(e_n)_{n=1}^{\infty} \subseteq Z^*$ the sequence of functionals given by $e_n(z) := z_n$ for every $z \in Z$.

\vspace{0.2cm}

M.D. Acosta observed that $(e_n)_{n=1}^{\infty}$ is a 1-unconditional normalized basic sequence of $Z^*$ and if $X(w)$ is the closed linear span of $(e_n)_{n=1}^{\infty}$, then $X(w)^*$ is isometric to $Z$ (see \cite[Lemma 2.1]{acosta1}). Moreover, it is obtained by using the Dvoretzky-Rogers theorem and defining a suitable operator on $c_{00} \subseteq X(w)$ that if $w \in \ell_2 \setminus \ell_1$, then the set $\NA(X(w), Y)$ is {\it not} dense in $\mathcal{L}(X(w), Y)$ for any infinite dimensional strictly convex Banach space $Y$ (see \cite[Theorem 2.3]{acosta1}).

In contrast with Propositions \ref{example1} and \ref{d-star}, under the assumptions mentioned in the previous paragraph, the following result shows that the norm of $X(w)$ may not be always compatible with the action of a group of permutations of the natural numbers.

\begin{lemma} \label{auxlemma} Let $G$ be a group of permutations of the natural numbers. Let $w = (w_n)_{n=1}^{\infty}$ be a non-increasing sequence of positive numbers such that $w \not\in \ell_1$ and $w_1 < 1$. Suppose that $w$ is a non-increasing sequence of positive numbers and that the norm of $X(w)$ is $G$-invariant. Then, for every $n\in\mathbb N$, we have that $\Orb(n)$ is finite.
\end{lemma}

\begin{proof}
Fix a natural number $n$. Assume that $\Orb(n)$ is infinite. Without loss of generality and to simplify the notation we will assume that $\Orb(n)=\{n,n+1,n+2,\ldots\}$. Then, we have that $\max\{ w_i:i\in \Orb(n)\}=w_n$. Let $k \in \mathbb{N} \cup \{0\}$ so that $w_{n+k}=w_n$ but $w_{n+k+1}<w_n$. Notice that we can fix $g\in G$ with $g(n)>n+k$ and fix $i > n+k$ so that $g(i)> n+k$. Then $w_{i}<w_n$ and $w_{g(i)}, w_{g(n)} <w_n$. 

Consider the point 
\begin{equation*} 
x:=e_i^*+ \sum_{j=n}^{n+k} e_j^* \in X(w)^* = Z,
\end{equation*} 
where $(e_j^*)$ is the sequence of coefficient functionals associated with $(e_j)$. 
By the $G$-invariance of the norm on $X(w)$, the norm on $Z$ is $G^*$-invariant as well. Thus, we have that 
\begin{align*}
	\Vert x\Vert 
	=\Big\Vert g \Big (e_i^* +\sum_{j=n}^{n+k} e_j^* \Big)\Big\Vert =\Big \Vert e_{g(i)}^*+ \sum_{j=n}^{n+k} e_{g(j)}^* \Big\Vert
\end{align*}
for $g \in G^*$. 
Since $ w_i < w_{n}$, we have that
\begin{align*}
\Vert x\Vert
&=\Big\Vert (1-w) \Big( e_i^*+ \sum_{j=n}^{n+k} e_j^*\Big)  \Big \Vert_\infty+\Big \Vert w \Big( e_i^*+ \sum_{j=n}^{n+k} e_j^*\Big) \Big\Vert_1\\
&=(1-w_i)+(w_i+(k+1) \cdot w_{n}) \\
&=1+(k+1) \cdot w_{n}.
\end{align*}

Now, let us write $\{g(i),g(n),g(n+1),\ldots,g(n+k)\}$ as $\{\alpha_1,\alpha_2,\ldots,\alpha_{k+2}\}$ so that $\alpha_1<\alpha_2<\cdots<\alpha_{k+2}$. This implies that 
\begin{equation} \label{eq1} 
w_{\alpha_{k+2}}= \min\{w_{g(i)},w_{g(n)},w_{g(n+1)},\ldots,w_{g(n+k)}\}. 
\end{equation} 
Notice that, using the new notation $\alpha_1<\alpha_2<\cdots<\alpha_{k+2}$, we have that 
\begin{equation*}
\Vert g(x)\Vert =\Big \Vert  (1-w)\Big( \sum_{j=1}^{k+2} e_{\alpha_j}^*\Big) \Big\Vert_\infty+\Big \Vert w \Big(  \sum_{j=1}^{k+2} e_{\alpha_j}^*\Big) \Big\Vert_1
\end{equation*}
and by (\ref{eq1}) this imples that 
\begin{equation*}
	\|g(x)\| = (1-w_{\alpha_{k+2}})+\sum_{r=1}^{k+2} w_{\alpha_r} = 1+\sum_{r=1}^{k+1} w_{\alpha_r}.
\end{equation*}
Moreover, note that $g(j) \geq n$ for every $j=n, \ldots, n+k$ since $\Orb (n) = \{n, n+1, \ldots\}$. Since $g(i)$ and $g(n)$ are bigger than $n+k$, we obtain that $\alpha_{k+1} \geq n+k+1$; hence $w_{\alpha_{k+1}} \leq w_{n+k+1} < w_n$. 
Consequently, we have 
\begin{equation*}
	\|g(x)\| = 1+\sum_{r=1}^{k+1} w_{\alpha_r} \leq 1 + k \cdot w_n + w_{\alpha_{k + 1}} < 1 + (k+1) \cdot w_n = \|x\|.
\end{equation*}
This contradicts the $G$-invariance of the norm of $X(w)$.
\end{proof}

It is clear by Lemma \ref{auxlemma} that $\mathcal{F}_{G, \mathbb{N}}$ must be infinite when the domain space is $X(w)$; hence we cannot expect an analogous result of Proposition \ref{example1} (1) or  Proposition \ref{d-star} (1) in this case. Nevertheless, we can show that $\NA_G (X(w), Y)$ is {\it not} dense in $\mathcal{L}_G (X(w), Y)$ under an extra condition on the weight sequence $w$.

\begin{proposition} \label{example3} Let $G$ be a group of permutations of the natural numbers. Suppose that a non-increasing sequence $w$ in $\ell_2 \setminus \ell_1$ satisfies that $w \vert_{\Orb(n)}$ is constant for each $n \in \N$ and that there exists $C >0$ such that $ |H| < C$ for every $H \in \mathcal F_{G,\mathbb N}$. 
If $Y$ is an infinite dimensional strictly convex Banach space, then 
\begin{equation*} 	
	\overline{\NA_G (X(w), Y)} \neq \mathcal{L}_G(X(w), Y).
\end{equation*} 
\end{proposition}

\begin{proof} By Lemma \ref{auxlemma}, we have that $\Orb(n)$ is finite for every natural number $n$. 
Consider an enumeration $H_1, H_2, \ldots$ of the different orbits with $|H_n| < C$ for every $n \in \N$. 
Without loss of generality, we may assume that 
\[
H_1=\{1,\ldots,{n_1}\},H_2=\{{n_1+1},\ldots,{n_2}\},\ldots,H_k=\{{n_{k-1}+1},\ldots,{n_k}\}, \ldots. 
\]
For each $n \in \N$, consider the finite dimensional subspace of $X(w)$ given by $X_n=\spann\{e_k:k\in H_n \}$ and a norm-one functional $f_n$ on $X_n$ that is invariant under the action of the group $H_n$ (considered as a group in $\mathcal{L} (X_n)$). 
Since $(w_{n_k}) \in \ell_2$, as a consequence of the Dvoretzky-Rogers theorem, there exists a normalized sequence of vectors $(y_k)_{k=1}^{\infty}$ in $Y$ so that the series $\sum_{k=1}^\infty {w}_{n_k} y_k$ converges unconditionally (see \cite[Theorem 1.2]{DJT}). Let us consider the $G$-invariant operator $T$ from $X(w)$ into $Y$ given by 
	\[
		T(e_j)=f_n (e_j) y_n
	\]
 where $j\in H_n$ for some $n \in \N$. 
	
	 Arguing as in \cite[Lemma 2.1]{acosta1} it can be shown that $B_{X(w)}=\overline{\text{co}} E$ where
	 \[
	 	E=\Big\{\theta_m(1-w_m)e_m+\sum_{i=1}^{n_k} \theta_iw_ie_i:\ m, k\in\mathbb N, \vert \theta_i\vert =1,\ \forall i\Big\}.
	 \]
Thus, to show that $T$ is bounded on $X(w)$ it is enough to show that $T$ is bounded on $E$. For this, consider the point
\begin{equation*} 
 x :=\theta_m(1-w_m)e_m+\sum_{i=1}^{n_k} \theta_iw_ie_i\in E. 
\end{equation*} 
Note that, by the $G$-invariance of the norm, setting $n_0=0$ we have that $w_{{n_j}+1+k}=w_{n_j+1}$ for $k=0,1,\ldots,n_{j+1}-n_j-1$. Therefore, 
\begin{equation*}
 \sum_{j=1}^{k} \sum_{i=n_{j-1} +1}^{n_j} \theta_i w_i f_{j} ( e_i) y_j = \sum_{j=1}^{k} w_{n_j} y_j \Big( \sum_{i=n_{j-1} +1}^{n_j} \theta_i f_j (e_i) \Big)
\end{equation*} 
and 
\begin{align*}
\Vert T(x)\Vert &\leq \| T(\theta_m (1-w_m)e_m) \| + \Big \| \sum_{j=1}^{k} \sum_{i=n_{j-1} +1}^{n_j} \theta_i w_i f_{j} ( e_i) y_j \Big\| \\
&\leq 1 + \Big \| \sum_{j=1}^{k} w_{n_j} y_j \Big( \sum_{i=n_{j-1} +1}^{n_j} \theta_i f_j (e_i) \Big) \Big\|.
\end{align*}  Note from the Bounded Multiplier Test \cite[Theorem 1.6]{DJT} that 
\[
\Big \| \sum_{j=1}^{k} w_{n_j} y_j \Big( \sum_{i=n_{j-1} +1}^{n_j} \theta_i f_j (e_i) \Big) \Big\| \leq C \|v\|, 
\]
where $v : Y^* \rightarrow \ell_1$ is the compact operator given by $v (y^*) = (y^* (w_{ } y_{n_m}))_{m\in\N} $ for every $y^* \in Y^*$ \cite[Theorem 1.9]{DJT}. It follows that
\[
\|T(x) \| \leq 1 + C \| v \|. 
\]
Thus, $T$ is a linear and continuous $G$-invariant operator.

Note from the proof of \cite[Theorem 2.3]{acosta1} that, for every norm-attaining operator $S$ from $X(w)$ into $Y$, there exists $N \in \N$ so that $S(e_k)=0$ for $k>N$. Thus, for any $n \in \N$ with $\min H_n >N$, we have that 
\[
\|T-S \| \geq \sup \{ \| (T-S)( x) \| : x \in B_{X_n} \} \geq \sup \{ | f_n (x) | : x \in B_{X_n} \} = 1. 
\]
Therefore, $T$ cannot be approximated by norm-attaining operators as desired.
\end{proof}

Notice that so far all the domains in the preceding examples were sequence spaces. Next, we would like to consider classical function spaces, such as $L_1 [0,1]$ and $C(K)$. J. Johnson and J. Wolfe \cite{JW} showed that there exists a compact metric space $K$ such that $\NA(L_1[0,1], C(K))$ is {\it not} dense in $\mathcal{L}(L_1[0,1], C(K))$, and W. Schachermayer observed that $\NA(L_1[0,1], C[0,1])$ is {\it not} dense in $\mathcal{L}(L_1[0,1], C[0,1])$ \cite{S}. Proposition \ref{L1-case} shows that an analogous result for $G$-invariant operators goes in the opposite direction.

\begin{proposition} \label{L1-case} Let $Y$ be any real or complex infinite dimensional Banach space with Schauder basis. Then, there exists a group $G \subseteq \mathcal{L}(L_1[0,1])$ such that every $G$-invariant operator from $L_1[0,1]$ into $Y$ is norm-attaining, that is, 
\begin{equation*}
\NA_{G}(L_1[0,1], Y) = \mathcal{L}_G(L_1[0,1], Y).
\end{equation*} 	
\end{proposition}

\begin{proof} Let $Y$ be any infinite dimensional Banach space over $\K=\R$ or $\K=\C$ with Schauder basis $(e_i)$ and consider $G \subseteq \mathcal{L}(L_1[0,1])$ to be the group induced by all measure preserving mappings $\varphi:[0,1] \rightarrow [0,1]$, that is, $G$ is the set
	\begin{equation*}
		\Big\{ R: L_1[0,1] \rightarrow L_1[0,1]: R(x) = x \circ \phi, \, \phi:[0,1] \rightarrow [0,1] \ \mbox{a measure-preserving mapping} \Big \} 
	\end{equation*}
	endowed with the topology induced by $\mathcal{L}(L_1[0,1])$ (see Section 5 of \cite{AGPZ}). Given $T \in \mathcal{L}_G(L_1[0,1], Y)$, we will prove that $T$ attains its norm. We can write 
	\begin{equation*}
		T(x) = \sum_{i=1}^{\infty} (T^* e_i^*)(x)e_i \ \ \ \ (x \in L_1[0,1]). 
	\end{equation*}
	Thus, $T^* e_i^* \in \mathcal{L}_G(L_1[0,1], \K)$ for every $i \in \N$ (see Proposition \ref{lemma-caract-dual-bidual}). It follows from the proof of \cite[Corollary 5.4]{AGPZ} that there exists a polynomial $P_i: \K \rightarrow \K$ such that
	\begin{equation*}
		(T^* e_i^*)(x) = P_i \Big( \int_0^1 x(t)\, dt \Big). 
	\end{equation*}
	Let us notice that, in this case, $P_i \in \mathcal{L}(\K) = \K^* = \K$ and write $P_i = \alpha_i \in \K$ for every $i \in \N$. Thus, 
	\begin{equation*}
		(T^* e_i^*)(x) = \alpha_i \int_0^1 x(t)\, dt
	\end{equation*}	
	for every $i \in \N$. It follows that 
	\begin{equation*}
		T(x) = \sum_{i=1}^{\infty} (T^* e_i^*)(x) e_i = \sum_{i=1}^{\infty} \Big( \alpha_i \int_0^1 x(t)\, dt \Big) e_i 
	\end{equation*} 
	for every $x \in L_1 [0,1]$. This, in particular, shows that 
	\[
	T(x) =  T \Big( \Big( \int_0^1 x(t) \, dt \Big) \mathds{1} \Big), 
	\]
	for every $x \in L_1 [0,1]$, where $\mathds{1} \in L_1[0,1]$ is the constant function given by $\mathds{1}(t) = 1$ for every $t \in [0,1]$. Let us define the mapping $\widetilde{T}: \K \rightarrow Y$ given by $\widetilde{T}(\lambda) := T(\lambda \cdot \mathds{1})$ for every $\lambda \in \K$. Then, on the one hand, we have that $\|\widetilde{T}\| \leq \|T\|$, and, on the other hand,
	\begin{equation*}
		\|\widetilde{T}\| \geq \Big\| \widetilde{T} \Big(\int_0^1 x(t) \, dt \Big) \Big\| = \Big\| T \Big( \Big( \int_0^1 x(t) \, dt \Big) \mathds{1} \Big) \Big\| = \|T(x)\| 
	\end{equation*}
	for every $x \in L_1[0,1]$ with $\|x\| \leq 1$. This shows that $\|\widetilde{T}\| = \|T\|$. Since $\widetilde{T}$ is defined on the finite dimensional space $\K$, it attains its norm at, let us say, $\lambda_0 \in B_{\K}$. Therefore, $T$ attains its norm at $\lambda_0 \cdot \mathds{1} \in B_{L_1[0,1]}$ as we wanted to prove.
\end{proof}

\begin{remark} J.J. Uhl \cite{uhl} proved that if $Y$ is any strictly convex Banach space and the set $\NA(L_1(\mu), Y)$ is dense in $\mathcal{L}(L_1(\mu), Y)$ (here we are assuming that $\mu$ is a finite measure), then $Y$ must satisfy the Radon-Nikod\'ym property. Nevertheless, Proposition \ref{L1-case} shows that even if $Y$ is an infinite dimensional strictly convex Banach space with Schauder basis which does \emph{not} have the Radon-Nikod\'ym property, we have that $\mathcal{L}_G(L_1[0,1], Y) = \NA_G(L_1[0,1], Y)$ when $G$ is the group induced by all measure preserving mappings on $[0,1]$. Therefore, Uhl's analogous result {\it does not} hold true in this context. 
\end{remark}

J. Lindenstrauss characterized when the set $\NA (C(K), Y)$ is dense in $\mathcal{L} (C(K), Y)$ for every Banach space $Y$. Indeed, he showed that, for a compact metric space $K$, the set $\NA (C(K), Y)$ is dense in $\mathcal{L} (C(K), Y)$ for every Banach space $Y$ if and only if $K$ is finite (see \cite[Proposition 2.(b)]{L}). However, the following result shows that this characterization \emph{does not} hold in the context of $G$-invariant norm-attaining operators when we assume that $Y$ has a Schauder basis. 

\begin{proposition} There exists a group $G \subseteq \mathcal{L}(C[0,1])$ such that, for every real or complex infinite dimensional Banach space with Schauder basis $Y$, we have that every $G$-invariant operator from $C[0,1]$ into $Y$ attains its norm, that is, 
\begin{equation*} 	
\NA_G(C[0,1], Y) = \mathcal{L}_G(C[0,1], Y).
\end{equation*} 
\end{proposition}

\begin{proof} Let $Y$ be an infinite dimensional Banach space over $\K=\R$ or $\K=\C$ with Schauder basis $(e_i)$. Let us consider the group $G \subseteq \mathcal{L}(C[0,1])$ of composition operators on $C[0,1]$ defined as 
	\begin{equation*}
		\Big\{ \gamma: C[0,1] \rightarrow C[0,1]: \gamma(f) = f \circ \phi \ \mbox{for some homeomorphism} \ \phi: [0,1] \rightarrow [0,1] \Big\}.
	\end{equation*}
We will prove that every $T \in \mathcal{L}_G(C[0,1], Y)$ is norm-attaining. Let us write 
	\begin{equation*}
		T(x) = \sum_{i=1}^{\infty} (T^* e_i^*)(x) e_i \ \ \ \ (x \in C[0,1]). 
	\end{equation*}
	Therefore, $T^* e_i^* \in \mathcal{L}_G(C[0,1], \K)$ for every $i \in \N$ (see Proposition \ref{lemma-caract-dual-bidual}). By adapting the proof of \cite[Theorem 3.1]{AGPZ}, there exists a symmetric function $F_i$ from $\K^2$ to $\K$ such that
	\begin{equation*}
		(T^* e_i^*)(x) = F_i(x(0), x(1))
	\end{equation*}
	for every $i \in \N$ with $x \in C[0,1]$. Define $\widetilde{T}: (\K^2, \|\cdot\|_{\infty}) \rightarrow Y$ by 
	\begin{equation*}
		\widetilde{T}(z, w) := \sum_{i=1}^{\infty} F_i(z, w) e_i \ \ \ \ ((z, w) \in \K^2).
	\end{equation*}
	So, $\widetilde{T}$ is a bounded linear operator defined on a finite dimensional space such that 
	\begin{align*}
		\|T\| &= \sup_{x \in B_{C[0,1]}} \Big\| \sum_{i=1}^{\infty} F_i(x(0), x(1)) e_i \Big\| \\
		&= \sup_{(z, w) \in B_{(\K^2,\|\cdot\|_{\infty})}} \Big\| \sum_{i=1}^{\infty} F_i(z, w) e_i \Big\| 
		= \sup_{(z, w) \in B_{(\K^2,\|\cdot\|_{\infty})}} \|\widetilde{T}(z, w)\|
		= \|\widetilde{T}\|. 	
	\end{align*} 
	Let $(z_0, w_0) \in B_{(\K^2,\|\cdot\|_{\infty})}$ be such that $\|\widetilde{T}\| = \|\widetilde{T}(z_0, w_0)\|$ and consider the point
	\begin{equation*}
		x_0 := z_0 \cdot f + w_0\cdot (1 - f), 
	\end{equation*}
	where $f(t) := 1 - t$ for every $t \in [0,1]$. Then, $x_0 \in B_{C[0,1]}$, $x_0(0) = z_0$, $x_0(1) = w_0$, and 
	\begin{eqnarray*}
		\|T(x_0)\| =  \Big\| \sum_{i=1}^{\infty} (T^* e_i^*)(x_0)e_i \Big\| 
		= \Big\| \sum_{i=1}^{\infty} F_i(z_0, w_0) e_i \Big\| = \|\widetilde{T}(z_0, w_0)\| = \|\widetilde{T}\| = \|T\|.
	\end{eqnarray*}
\end{proof}

\subsection{Properties $\beta$ of Lindenstrauss and $\alpha$ of Schachermayer}

A Banach space $X$ (respectively, $Y$) has \emph{property A} (respectively, \emph{property B}) if the set $\NA (X, Z)$ is dense in $\mathcal{L} (X,Z)$ (respectively, $\NA (Z, Y)$ is dense in $\mathcal{L} (Z, Y)$) for every Banach space $Z$ (see \cite{L}). It is shown that $\ell_1$ has property A while $c_0$ and $\ell_\infty$ have \emph{property $\beta$}, which is a strengthening of property B. Motivated by the notion of property $\beta$, W. Schachermayer \cite{S} considered \emph{property} $\alpha$ and showed this property implies property $A$. Besides, it was rather recently introduced \emph{property quasi-$\alpha$} and \emph{property quasi-$\beta$} and observed that those properties are, in fact, new sufficient conditions for properties A and B, respectively (for more details we recommend \cite{AAP, CS}). 
Our aim in this section is to discuss the analogues of properties $\alpha$ and $\beta$ (as well as quasi-$\alpha$ and quasi-$\beta$) in the context of $G$-invariant norm-attaining operators. We will start by generalizing property $\alpha$ and its strengthening property quasi-$\alpha$.

\begin{definition} \label{propertyalpha} Let $X$ be a Banach space and $G \subseteq \mathcal{L} (X)$ be a group of isometries. 
\begin{itemize}
\item[(a)]
$X$ is said to have {\it property $\alpha$-$G$} if there exist $\{(x_\lambda, x_\lambda^*): \lambda\in \Lambda\}\subseteq X \times X^*$ and a real number $0\leq\rho < 1$ satisfying that
	\begin{itemize}
		\setlength\itemsep{0.3em}
		\item[(i)] $x_{\lambda}^*$ is $G$-invariant for every $\lambda \in \Lambda$,
		\item[(ii)] $x_{\lambda}^*(x_{\lambda}) = \Vert x_{\lambda}^*\Vert =\Vert x_{\lambda}\Vert =1$ for every $\lambda \in \Lambda$,
		\item[(iii)] $|x_{\lambda}^*(x_{\mu})| \leq \rho$ for every $\lambda,\mu\in \Lambda$, $\lambda \not= \beta$, and
		\item[(iv)] $B_X$ is the closed absolutely convex hull of the set $\{g(x_\lambda) : \lambda\in\Lambda, g\in G\}$. 
		\end{itemize}
	
\vspace{0.2cm} 	
		
	\item[(b)] $X$ is said to have {\it property quasi-$\alpha$-$G$} if there exist $A := \{x_{\lambda} \in S_X: \lambda \in \Lambda\} \subseteq S_X$, $A^* := \{x_{\alpha}^* \in S_{X^*}: \lambda \in \Lambda\}$, and a function $\rho: A \rightarrow \R$ such that
	\begin{itemize}
		\setlength\itemsep{0.3em}
		\item[(i)] $x_{\lambda}^*$ is $G$-invariant for every $\lambda \in \Lambda$, 
		\item[(ii)] $x_{\lambda}^*(x_{\lambda}) = 1$ for every $\lambda \in \Lambda$, 
		\item[(iii)] $|x_{\lambda}^*(x_{\mu})| \leq \rho(x_{\lambda}) < 1$ for every $x_{\lambda}^* \in A^*$, $x_{\lambda}, x_{\mu} \in A$, $\lambda \not= \mu$, and 
		\item[(iv)] for every $e \in \ext{B_{X^{**}}}$, there is a subset $A_e \subseteq A$ and a scalar $t$ with $|t| = 1$ such that $te \in \overline{Q(G(A_e))}^{w^*}$ and $r_e := \sup \{ \rho(x): x \in A_e\} < 1$, where $Q: X \hookrightarrow X^{**}$ is the canonical embedding of $X$ in $X^{**}$.
	\end{itemize}
	\end{itemize} 
\end{definition}

Let us briefly notice that in Definition \ref{propertyalpha}; if we assume that $G$ is the trivial group consisting of only the identity mapping, then property $\alpha$-$G$ coincides with property $\alpha$ of Schachermayer as in \cite{S} and property quasi-$\alpha$-$G$ coincides with property quasi-$\alpha$ as in \cite{CS}. Let us first show the following example of a space $X$ satisfying property $\alpha$-${{G}}$ for a group $G$ of permutations of the natural numbers.

\begin{example} \label{example-ell1} Set $X = \ell_1$ and consider $G$ to be a group of permutations on $\N$. We prove that $X$ satisfies property $\alpha$-${{G}}$ with constant $\rho = 0$ where $G$ is considered as a group in $\mathcal{L} (X)$. In order to do this, let us construct the set $\{(x_\lambda, x_\lambda^*): \lambda\in \Lambda\}\subseteq X \times X^*$. Denote by $(e_n)_{n \in \N}$ the canonical basis of $X$ and set $n_0=1$. Let $x_1 := e_1$ and define $\N_1 := \N \setminus \{g(1): g \in G\}$. Now, take $n_1 := \min \N_1$, define $x_2 := e_{n_1}$ and consider $\N_2 := \N_1 \setminus \{g(n_1): g \in G\}$. Inductively, we construct the sequences $(x_k) \subseteq S_X$ and $(\N_k) \subseteq \N$ with $n_k := \min \N_k$, $x_k := e_{n_{k-1}}$, and $\N_k := \N_{k-1} \setminus \{g(n_{k-1}): g \in G\}$. To show that $X$ satisfies property $\alpha$-${G}$ we consider two cases.
		
	\vspace{0.2cm}
	\noindent
	{\bf Case 1}: Suppose that $\N_k \not= \emptyset$ for every $k \in \N$. 
	\vspace{0.2cm}

	Given $k \in \N$, let us define $x_k^*: \ell_1 \rightarrow \K$ by
	\begin{equation*}
		x_k^*(e_n) := \max \left\{ e_{g(n_{k-1})}^* (e_n): g \in G \right\}
	\end{equation*}	
	for every $n \in \N$. Notice that, for every $k \in \N$, we have that $x_k^* \in B_{\ell_{\infty}}$ and also that 
	\begin{equation*}
		x_k^*(x_k) = x_k^*(e_{n_{k-1}}) = \max \left\{ e_{g(n_{k-1})}^*(e_{n_{k-1}}): g \in G \right\} = 1. 
	\end{equation*}	
	In particular, $\|x_k^*\| = 1$ for every $k \in \N$ and this gives condition (ii). Let us prove now condition (iii). Assume that $l \not= k$ and observe that
	\begin{equation} \label{emptyset} 
		\left\{ g(n_{k-1}): g \in G \right\} \cap \left\{ n_{l-1} \right\} = \emptyset.
	\end{equation}
Therefore, if $l \not= k$, then 
	\begin{equation*}
		x_k^*(x_l) = x_k^*(x_{n_{l-1}}) = \max \left\{ e_{g(n_{k-1})}^*(e_{n_{l-1}}): g \in G \right\} \stackrel{(\ref{emptyset})}{=} 0
	\end{equation*}
	and this proves condition (iii). To prove condition (iv), let us fix $n \in \N$. As $n_k > n_{k-1}$ for every $k \in \N$, there must exist $k_0 \in \N$ such that $n \not\in \N_{k_0}$. Let us consider $k_0 \in \N$ to be such that $n \not\in \N_{k_0}$ but $n \in \N_{k_0-1}$. Therefore, $n = g(n_{k_0-1})$ for some $g \in G$ and this shows that
	\begin{equation*}
		\left\{ {g}(x_n): n \in \N, {g} \in {G} \right\} = \{ e_n : n \in \N \}
	\end{equation*}	
	and, in particular, we have condition (iv). Finally, let us prove that each $x_k^*$ is ${G}$-invariant for every $k \in \N$. Indeed, let $h \in {G}$ and fix $k \in \N$. Then, we have that 
	\begin{eqnarray*}
		x_k^*(h(e_n)) = x_k^*(e_{h(n)}) = \max \left\{ e_{(h^{-1} \circ g)(n_{k-1})}^*(e_n): g \in G \right\} = x_k^*(e_n)
	\end{eqnarray*}
	and this gives condition (i). Therefore, $X$ has property $\alpha$-${G}$.

	\vspace{0.2cm}
	\noindent
	{\bf Case 2}: Suppose now that there is $N \in \N$ such that $\N_N = \emptyset$ and $\N_{N-1} \not= \emptyset$. 
	\vspace{0.2cm}
	
	Define, for every $k = 1, \ldots, N$, the functional
	\begin{equation*}
		x_k^*(e_n) := \max \left\{ e_{g(n_{k-1})}^*(e_n): g \in G \right\}
	\end{equation*} 
	and we follow the proof of Case 1 to show that $X$ has property $\alpha$-${G}$.
\end{example}

We now present the counterpart for \cite[Theorem 2.1.(i)]{CS} in the context of invariant groups. In what follows, we denote by $\overline{\aco}(A)$ the closed absolutely convex hull of a set $A$.

\begin{proposition} \label{alpha-quasialpha} Let $X$ be a Banach space and $G \subseteq \mathcal{L}(X)$ be a group of isometries. If $X$ has property $\alpha$-$G$, then it also has property quasi-$\alpha$-$G$. 
\end{proposition}

\begin{proof} We follow the arguments from \cite[Theorem 2.1.(i)]{CS}. Suppose that $X$ satisfies property $\alpha$-$G$ with the family $\{(x_{\lambda}, x_{\lambda}^*): \lambda \in \Lambda\} \subseteq S_X \times S_{X^*}$ and the real number $0 \leq \rho < 1$. Then, $X$ satisfies property quasi-$\alpha$-$G$ with $A_e = A$ for every $e \in \ext{B_{X^{**}}}$ and $\rho(x_{\lambda}) := \rho$ for every $\lambda \in \Lambda$. Indeed, we consider the set $Q(G(A))$, where $Q: X \hookrightarrow X^{**}$ is the canonical embedding of $X$ in $X^{**}$. Since $\overline{\aco{}}(G(A)) = B_{X^{**}}$, we have that $\overline{\co{}}^{w^*} (\mathbb{T}Q(G(A))) = B_{X^{**}}$, where $\mathbb{T} = \{z \in \C: |z| = 1\}$. Since $B_{X^{**}}$ is convex and $w^*$-compact, by Krein–Milman theorem, we have that $\ext{B_{X^{**}}} \subseteq \mathbb{T} Q(G(A))$ and this allows us to take $A_e := A$ for every $e \in \ext{B_{X^{**}}}$ and $\rho(x_{\lambda}) := \rho$ for every $\lambda \in \Lambda$.
\end{proof}

Analogously to what we have done before, the notions of property $\beta$ and property quasi-$\beta$-$G$ for a group $G \subseteq \mathcal{L} (X)$ are given as follows.

\begin{definition} \label{propertybetaG} Let $Y$ be a Banach space and $H \subseteq \mathcal{L} (Y)$ be a group of isometries. 

\begin{itemize}	
\item[(a)] $Y$ is said to have {\it property $\beta$-$H$} if there exist $\{y_{\lambda}: \lambda \in \Lambda\} \subseteq S_Y$, $\{y_{\lambda}^*: \lambda \in \Lambda\} \subseteq S_{Y^*}$, and a number $0 \leq \rho < 1$ such that
	\begin{itemize}	 
		\item[(i)] $y_{\lambda}$ is an $H$-invariant point for every $\lambda \in \Lambda$, 
		\item[(ii)] $y_{\lambda}^*(y_{\lambda}) = 1$ for every $\lambda \in \Lambda$, 
		\item[(iii)] $|y_{\lambda}^*(y_{\mu})| \leq \rho$ for every $\lambda \not= \mu$, $\lambda, \mu \in \Lambda$, and
		\item[(iv)] for every $y \in Y$, we have $\|y\| = \sup \{|y_{\lambda}^*(g(y))|: \lambda \in \Lambda, g \in H\}$. 
	\end{itemize}

\vspace{0.2cm} 

		\item[(b)] $Y$ is said to have {\it property quasi-$\beta$-$H$} if there exist $A = \{y_{\lambda}: \lambda \in \Lambda\} \subseteq S_Y$, $A^* = \{y_{\lambda}^*: \lambda \in \Lambda\} \subseteq S_{Y^*}$, and a function $\rho: A^* \rightarrow \R$ such that 
	\begin{itemize}
		\setlength\itemsep{0.3em}
		\item[(i)] $y_{\lambda}$ is an $H$-invariant point for every $\lambda \in \Lambda$,
		\item[(ii)] $y_{\lambda}^*(y_{\lambda}) = 1$ for every $\lambda \in \Lambda$, 
		\item[(iii)] $|y_{\mu}^*(y_{\lambda})| \leq \rho (y_{\lambda}^*) < 1$ whenever $\lambda \not= \mu$ for every $\lambda, \mu \in \Lambda$, and
		\item[(iv)] for every $e^* \in \ext{B_{Y^*}}$, there exists a subset $A_{e^*} \subseteq A^*$ and $t \in \mathbb{T}$ such that $te^* \in \overline{H^*(A_{e^*})}^{w^*}=\overline{\{S^*(A_{e^*}):S\in H\}}^{w^*}$ and $r_{e^*} := \sup \{ \rho(y^*): y^* \in A_{e^*}\} < 1$. 	
	\end{itemize} 
	\end{itemize} 
	If $E$ is a subspace of $X^*$, we say that $X$ has {\it property $\beta$-$H$ induced by $E$} if we can find $\{y_{\lambda}^*: \lambda \in \Lambda\}$ in $E$ satisfying the previous conditions in (a). Similarly, if $F$ is a subspace of $Y^*$ and if we can find $A^*$ so that $A^* \subseteq F$, where $A^*$ is the set in the above definition in (b), then we say that $Y$ has {\it property quasi-$\beta$-$H$ induced by $F$}.

\end{definition}

As in the case of properties $\alpha$ and quasi-$\alpha$, let us notice that if, in Definition \ref{propertybetaG}, we assume that $H$ is the trivial group consisting of only the identity mapping, then property $\beta$-$H$ coincides with property $\beta$ of Lindenstrauss as in \cite{L} and property quasi-$\beta$-$H$ coincides with property quasi-$\beta$ as in \cite{AAP}. 
Analogously to Proposition \ref{alpha-quasialpha}, we have the following result (see the comments just after \cite[Definition 1]{AAP}).

\begin{proposition} Let $Y$ be a Banach space and $H \subseteq \mathcal{L}(Y)$ be a group of isometries. If $Y$ has property $\beta$-$H$, then it also has property quasi-$\beta$-$H$.	
\end{proposition}

Next, we present the following duality results.

\begin{proposition} \label{dual-induced} Let $X$ be a Banach space and $G \subseteq \mathcal{L}(X)$ be a group of isometries. 
\begin{itemize}
		\setlength\itemsep{0.3em}
\item[(a)]
$X$ has property $\alpha$-$G$ if and only if $X^*$ has property $\beta$-$G^*$ induced by $X$.
\item[(b)] 
$X$ has property quasi-$\alpha$-$G$ if and only if $X^*$ has property quasi-$\beta$-$G^*$ induced by $X$
\end{itemize}
\end{proposition}

\begin{proof} 
(a) can be obtained similarly to \cite[Proposition 1.4]{S}. For (b), suppose that $X$ has property quasi-$\alpha$-$G$ and $Y = X^*$. Consider the sets $A =\{x_{\lambda} \in S_X: \lambda \in \Lambda\}$ and $A^* = \{x_{\alpha}^* \in S_{X^*}: \lambda \in \Lambda\}$, and the function $\rho: A \rightarrow \R$ from (b) in Definition \ref{propertyalpha}. Let us notice that, since $x_{\lambda}^*$ is $G$-invariant we have that
\begin{equation*}
	g^*(x_{\lambda}^*)(x) = x_{\lambda}^*(g(x)) = x_{\lambda}^*(x) 
\end{equation*} 
for every $x \in X$ and every $g\in G$. Therefore, $x_{\lambda}^*$ is a $G^*$-invariant point. Moreover, let $e^* \in \ext{B_{Y^*}} = \ext{B_{X^{**}}}$. Then, there exists $A_{e^*} \subseteq A$ and $t \in \mathbb{T}$ such that 
\begin{equation*}
	te^* \in \overline{Q(G(A_{e^*}))}^{w^*} \ \ \ \ \mbox{and} \ \ \ \ r_{e^*} = \sup \{ \rho(x): x \in A_{e^*}\} < 1.
\end{equation*}
Finally, we may define $\rho_Y: Q(A) \rightarrow \R$ by $\rho_{Y} (Q(x_{\lambda})) = \rho(x_{\lambda})$ for every $\lambda \in \Lambda$. This is enough to prove (b). 
\end{proof}

\begin{example} As in Example \ref{example-ell1}, let us consider a group of permutations $G$ in $\N$ and keep the notation $G$ for the induced group in $\mathcal{L} (\ell_1)$. 
	Note that ${G}^* =\{ g^* : g \in G\}$ is a group in $\mathcal{L}(\ell_{\infty})$ and that 
	\begin{equation*}
		g^* (z) = (z_{g^{-1} (n)})_{n=1}^{\infty}
	\end{equation*}	
	for every $z \in \ell_\infty$ and $g \in G$. Example \ref{example-ell1} shows that $\ell_1$ satisfies property $\alpha$-${G}$ and, therefore, by Proposition \ref{dual-induced}, $\ell_{\infty}$ satisfies property $\beta$-${G}^*$ (induced by $\ell_1$) with $\rho = 0$. 
\end{example}

\subsection{Denseness of the $G$-invariant norm-attaining operators} \label{sectionadj}

In this section, we give some invariant group versions of the main results from \cite{AAP, CS, L, Z} making use of all the tools we have obtained so far. We start with the following result due to Lindenstrauss and, for its proof, we invoke Propositions \ref{lemma-convergence} and \ref{lemma-caract-dual-bidual}.

\begin{theorem}[Lindenstrauss] \label{lindenstrauss} Let $X, Y$ be Banach spaces and $G \subseteq \mathcal{L}(X)$ be a group of isometries. Then, the set
	\begin{equation*}
\{ T \in \mathcal{L}_G(X, Y): T^{**} \in \NA_{G^{**}}(X^{**}, Y^{**}) \}
	\end{equation*}
	is dense in $\mathcal{L}_G(X, Y)$. 
\end{theorem}

\begin{proof} Let $T \in \mathcal{L}_G(X, Y)$ with $\|T\| = 1$ and $\eps \in (0, 1/3)$ be given. Consider $(\eps_k)_{k=1}^{\infty} \subseteq \R^+$ to be such that
	\begin{equation*}
		2 \sum_{i=1}^{\infty} \eps_i < \eps, \ \ \ 2 \sum_{i=k+1}^{\infty} \eps_i < \eps_k^2 \ \ \ \mbox{and} \ \ \ \eps_k < \frac{1}{10k}, \ \forall k \in \N. 
	\end{equation*} 
	Let us construct inductively $(T_k)_{k=1}^{\infty}, (x_k)_{k=1}^{\infty}$, and $(x_k^*)_{k=1}^{\infty}$ to be such that
	\begin{itemize}
		\setlength\itemsep{0.3em}
		\item[(i)] $T_1:= T \in \mathcal{L}_G(X, Y)$,
		\item[(ii)] $\|T_k(x_k)\| > \|T_k\| - \eps_k^2$ with $\|x_k\| = 1$,
		\item[(iii)] $x_k^*(T_k(x_k)) = \|T_k(x_k)\|$ with $\|x_k^*\| = 1$, and
		\item[(iv)] $T_{k+1}(x) := T_k(x) + \eps_k x_k^*(T_k(x)) T_k(x_k)$ for every $k \in \N$ and $x \in X$.
	\end{itemize}

	It is clear that every $T_k$ is $G$-invariant since $T_1$ is $G$-invariant. By Lindenstrauss' argument (see the proof of \cite[Theorem 1]{L}), we have that $(T_k)_{k=1}^{\infty}$ converges (in the strong operator topology) to some $T_{\infty} \in \mathcal{L}(X, Y)$ such that $\|T_{\infty} - T\| < \eps$ and $T_{\infty}^{**} \in \NA(X^{**}, Y^{**})$. By Proposition \ref{lemma-convergence}, $T_{\infty} \in \mathcal{L}_G(X, Y)$ and by Proposition \ref{lemma-caract-dual-bidual}, $T_{\infty}^{**}$ is $G^{**}$-invariant.
\end{proof}

Next result extents Theorem \ref{lindenstrauss}. As before, we use Proposition \ref{lemma-convergence} and Lemma \ref{lemma-caract-dual-bidual}.

\begin{theorem}[Zizler] \label{zizler} Let $X, Y$ be Banach spaces and let $G \subseteq \mathcal{L}(X)$ be a group of isometries. Then, the set
	\begin{equation*}
 \{ T \in \mathcal{L}_G(X, Y): T^* \in \NA(Y^*, X^*) \}
	\end{equation*}
	is dense in $\mathcal{L}_G(X, Y)$.
\end{theorem}

\begin{proof} Given $T \in \mathcal{L}_G(X, Y)$ with $\|T\| = 1$ and $\eps > 0$, as in Lindenstrauss argument, Zizler's proof in \cite{Z} relies on inductive construction of a sequence $(T_k)_{k=1}^{\infty} \subseteq \mathcal{L}_G(X, Y)$ which converges (in the strong operator topology) to some $T_{\infty} \in \mathcal{L}(X, Y)$ so that $T_{\infty}^* \in \NA(Y^*, X^*)$. By the same argument as in Theorem \ref{lindenstrauss}, each $T_k$ is $G$-invariant and, therefore, so is $T_{\infty}$ by Proposition \ref{lemma-convergence}. 
\end{proof}

In Theorem \ref{AcostaAguirrePaya} below we assume that the group of isometries $G \subseteq \mathcal{L}(X)$ is {\it compact} as its proof makes use of a version of the Bishop-Phelps theorem for group invariant mappings (see \cite[Theorem 3]{F}). 

\begin{theorem}[Acosta, Aguirre, Pay\'a] \label{AcostaAguirrePaya} Let $X$ be a Banach space and $G \subseteq \mathcal{L}(X)$ be a compact group of isometries. If $Y$ has property quasi-$\beta$-$H$ for some $H \subseteq \mathcal{L}(Y)$, then 
	\begin{equation*}
		\{ T \in \NA_G(X, Y): T^* \in \mathcal{L}_{H^*} (Y^*, X^*) \} 
	\end{equation*}
	is dense in $\{T \in \mathcal{L}_G(X, Y): T^* \in \mathcal{L}_{H^*}(Y^*, X^*)\}$.
\end{theorem}

Notice that in the particular case of $G = \{ \id_X \} \subseteq \mathcal{L} (X)$ and $H = \{\id_Y\} \subseteq \mathcal{L}(Y)$ in Theorem \ref{AcostaAguirrePaya}, we obtain the result \cite[Theorem 2]{AAP}. 

\begin{proof}[Proof of Theorem \ref{AcostaAguirrePaya}] We follow the arguments from \cite[Theorem 2]{AAP}. Let $T \in \mathcal{L}_G(X, Y)$ be such that $T^* \in \mathcal{L}_{H^*} (Y^*, X^*)$. By Theorem \ref{zizler}, we may assume that $T^* \in \NA(Y^*, X^*)$ with $\|T\| = 1$. By \cite[Theorem 5.8]{Lima}, $T^*$ attains its norm at some $e^* \in \ext{B_{Y^*}}$. Thus, there exists $A_{e^*} \subseteq A^*$ and $t \in \mathbb{T}$ such that
	\begin{equation*}
		te^* \in \overline{H(A_{e^*})}^{w^*} \ \ \ \ \mbox{and} \ \ \ \ r := \sup \{ \rho(y^*): y^* \in A_{e^*}\} < 1.
	\end{equation*}
	Fix $0 < \delta < \frac{\eps}{2}$ such that
	\begin{equation*}
		1 + r \left( \frac{\eps}{2} + \delta \right) < \left(1 + \frac{\eps}{2} \right)(1 - \delta).
	\end{equation*}
	Let us notice now that for every $y \in Y$, we have that 
	\begin{equation*}
		\|y\| = \sup \{ |(h^* y_{\lambda}^*)(y)|: h \in H, \lambda \in \Lambda \} 
	\end{equation*}
	and
	\begin{equation*}
		\|T^*\| = \sup \{ \|T^*(h^* y_{\lambda}^*)\|: h \in H, \lambda \in \Lambda\} = \sup \{ \|T^*(y_{\lambda}^*)\|: \lambda \in \Lambda \},	
	\end{equation*}
	so we can take $\lambda_0 \in \Lambda$ to be such that 
	\begin{equation*} 
		\|T^*(y_{\lambda_0})\| > 1 - \delta. 
	\end{equation*} 
	Since $G$ is compact, there exists a $G$-invariant norm-attaining functional $x^* \in X^*$ such that $\|x^*\| = \|T^* y_{\lambda_0^*}\| > 1 - \delta$ and $\|x^* - T^* y_{\lambda_0}^*\| < \delta$ (see \cite[Theorem 3]{F}). Finally, define $S \in \mathcal{L}(X, Y)$ by 
	\begin{equation*}
		S(x) := T(x) + \left[ \left(1 + \frac{\eps}{2} \right) x^*(x) - T^* y_{\lambda_0}^*(x) \right] y_{\lambda_0} \ \ \ (x \in X). 
	\end{equation*}
	Then, $S \in \mathcal{L}_G(X, Y)$ and $S^* \in \mathcal{L}_{H^*} (Y^*, X^*)$ since $y_{\lambda_0}$ is an $H$-invariant point. Arguing as in \cite[Theorem 2]{AAP}, we have that $S \in \NA_G (X, Y)$ and $\|S - T\| < \eps$ as we wanted. 
\end{proof} 

Theorem \ref{AcostaAguirrePaya}, in particular, shows that the classical property quasi-$\beta$ of $Y$ is a sufficient condition for the set $\NA_G (X,Y)$ to be dense in $\mathcal{L}_G (X,Y)$ for every Banach space $X$ and for every compact group $G \subseteq \mathcal{L} (X,Y)$.

\begin{corollary}\label{cor:AAP}
Let $X$ be a Banach space and $G \subseteq \mathcal{L}(X)$ be a compact group of isometries. If $Y$ has property quasi-$\beta$, then $\NA_G(X, Y)$ is dense in $\mathcal{L}_G(X, Y)$. 
\end{corollary} 

As a counterpart to Theorem \ref{AcostaAguirrePaya}, we now present the following result on property quasi-$\alpha$-$G$ (see Definition \ref{propertyalpha}) but {\it without} the assumption that the group $G$ is compact. This is a generalisation of \cite{CS}.

\begin{theorem}\label{choisong} 
	Let $X$ be a Banach space and $G \subseteq \mathcal{L} (X)$ be a group of isometries. If $X$ satisfies property quasi-$\alpha$-$G$, then $\NA_G (X,Y)$ is dense in $\mathcal L_G(X,Y)$ for every Banach space $Y$.
\end{theorem}

\begin{proof}
	We follow the arguments from \cite[Proposition 2.1]{CS}. Fix a Banach space $Y$, $T \in \mathcal{L}_G (X,Y)$, and $\eps >0$. By Theorem \ref{lindenstrauss}, we may assume that $T^{**}$ attains its norm. 
Moreover, \cite[Theorem 5.8]{Lima} guarantees that $T^{**}$ attains its norm at some extreme point $e\in B_{X^{**}}$. Now, condition (iii) of property quasi-$\alpha$-$G$ ensures us that there exists some scalar $t$ of modulus one and a set $A_e$ so that 
	\begin{equation*} 
	te\in \overline{Q\left(G\left(A_e\right)\right)}^{w^*} \ \ \ \mbox{and} \ \ \ r_e = \sup\{\rho(x) : x \in A_e\} < 1.
\end{equation*} 		
	Now, fix $\delta >0$ so that 
	\begin{equation*}\label{eq:delta} 	
		(1+\eps)(1-\delta)>1+\eps r_e. 
	\end{equation*} 	
	Since $\Vert T^{**}(te)\Vert =\Vert T\Vert$ and $te\in \overline{Q\left(G\left(A_e\right)\right)}^{w^*}$ we can find $x_{\lambda_0}\in A_e$ so that $\Vert T (x_{\lambda_0})\Vert\geq 1-\delta$. If we define the operator 
	\begin{equation*}\label{eq:S}
		S(x)=T(x)+\eps x^*_{\lambda_0}(x)T(x_{\lambda_0}) \ \ \ (x \in X)
	\end{equation*} 
	then we have that $\Vert S(x_{\lambda_0})\Vert\geq (1+\eps)(1-\delta)$ and for every $\lambda\ne \lambda_0$, $\Vert S( x_{\lambda})\Vert\leq 1+\eps \rho(x_{\lambda_0}) \leq 1+\eps r_e$. Therefore, $S$ attains its norm at $x_{\lambda_0}$ and $\Vert S-T\Vert \leq \eps$. Moreover, $S$ is $G$-invariant and we are done.
\end{proof}





\subsection{The group invariant Bishop-Phelps property}

A subset $A$ of a Banach space $X$ is said to be {\it dentable} if for every $\eps > 0$, there is a point $x \in A$ such that $x \not\in \overline{\co} (A \setminus B(x, \eps))$ where $B(x, \eps) := \{y \in X: \|x - y\| < \eps\}$.
If every non-empty bounded subset of $X$ is dentable, then $X$ is said to be dentable. In \cite{B}, J. Bourgain provided one of the most remarkable results in norm-attaining theory by showing that the Radon-Nikod\'ym property is equivalent to the Bishop-Phelps property. We recall that a Banach space $X$ satisfies the {\it Bishop-Phelps property} if for every Banach space $Y$ and for every non-empty closed, absolutely convex bounded set $C$, the set of operators from $X$ into $Y$ that attain their supremum on $C$ is norm dense in $\mathcal L(X,Y)$.

Our goal here is to show that the Radon-Nikod\'ym property is also a sufficient condition to ensure a group invariant version of the Bishop-Phelps property.

\begin{definition}
	Let $X$ be a Banach space and $G$ a group in $\mathcal{L} (X)$. We say that $X$ has the \emph{$G$-Bishop-Phelps property} if for every Banach space $Y$ and for every non-empty closed, absolutely convex bounded set $C$ that is $G$-invariant, the set of $G$-invariant operators from $X$ into $Y$ that attain their supremum on $C$ is norm dense in $\mathcal L_G(X,Y).$
\end{definition}

 Our aim is to prove the following results that give us a sufficient condition to satisfy the $G$-Bishop-Phelps property, generalising (partly) the well known characterisation of the Bishop-Phelps property  obtained by Bourgain \cite{B}.

\begin{theorem} \label{RNP} Let $X$ be a Banach space with the Radon-Nikod\'ym property and $G \subseteq \mathcal{L}(X)$ be a compact group of isometries. Then, $X$ has the $G$-Bishop-Phelps property. 
\end{theorem}

Theorem \ref{RNP} will follow as an immediate consequence of Theorem \ref{theorem-bourgain} below and its proof is essentially included in the next lemma. To prove the lemma, we follow the arguments from \cite[Lemma 4]{B} as well as some of the tools we have studied about $G$-invariant operators in the previous sections. Given a bounded subset $B$ of $X$ and $T \in \mathcal{L} (X, Y)$, we set $\|T\|_B$ to denote $\sup \{ \| T(x) \| : x \in B \}$.

\begin{lemma}[Bourgain] \label{lemma-bourgain} Let $X$ be a real Banach space and $G \subseteq \mathcal{L}(X)$ a compact group. Let $B$ be a non-empty closed, absolutely convex, and $G$-invariant subset contained in $B_X$. Assume that every subset of $B$ is dentable. Let $Y$ be a Banach space. Given $\eps > 0$, define the set 
	\begin{equation*}
		A_{\eps} := \Big\{ T \in \mathcal{L}_G(X, Y): S(T, \eta) \subseteq B(p, \eps) \cup B(-p, \eps) \ \mbox{for some} \ \eta > 0 \ \mbox{and} \ p \in X \Big\},
	\end{equation*}
	where 
	\begin{equation*}
		S(T, \eta) := \{ x \in B \cap X_G: \|T(x)\| > \|T\|_B - \eta \}. 
	\end{equation*}
	Then, $A_{\eps}$ is dense in $\mathcal{L}_G(X, Y)$. Moreover, if $\delta > 0$ and $S \in \mathcal{L}_G(X, Y)$, then there exists $T \in A_{\eps}$ such that 
	\begin{equation*}
		\|S - T\| < \delta \ \ \ \ \mbox{and} \ \ \ \ S - T \ \mbox{is finite rank}. 
	\end{equation*}
\end{lemma}

\begin{proof} Assume to the contrary that there exist $\eps > 0$, $\delta \in (0, 1/2)$, and $S \in \mathcal{L}_G(X, Y)$ with $\|S\|_B = 1$ such that if $T \in \mathcal{L}_G(X, Y)$ satisfies $\|S - T\| < \delta$ and $S - T$ is finite rank, then $T \not\in A_{\eps}$.

	For each $n \in \N$, we define the set $V_n$ consisting of all elements $x \in B \cap X_G$ such that there exists $T \in \mathcal{L}_G(X, Y)$ such that 
	\begin{equation*} 
		\|T(x)\| \geq \|T\|_B - \frac{\delta^2}{4^n}, \ \ \ \ \|S - T\| \leq \delta \left(1 - \frac{1}{2^n} \right), \ \ \ \ \mbox{and} \ \ \ S - T \ \mbox{is finite rank}.
	\end{equation*}
	Observe from Proposition \ref{lemma1} and Proposition \ref{norming} that $V_n$ is well-defined and it is non-empty. 
	
	\vspace{0.2cm}
	\noindent
	{\bf Claim}: If $z \in V_n$ and $p \in X$, then 
	\begin{equation*}
		\dist \Big( z, \co (V_{n+1} \setminus B(p, \eps)) \Big) < \frac{\delta}{2^{n-6}}.
	\end{equation*}
	\vspace{0.2cm} 	
By contradiction, assume that
	\begin{equation*}
		\dist \Big( z, \co (V_{n+1} \setminus B(p, \eps)) \Big) \geq \frac{\delta}{2^{n-6}} > \frac{r}{2^n},
	\end{equation*}
	where $2^5 \delta < r < 2^6 \delta$. Notice that the set $\co (V_{n+1} \setminus B(p, \eps))$ is $G$-invariant. By using Proposition \ref{functional} applied to the set $\co (V_{n+1} \setminus B(p, \eps))$, to $\frac{r}{2^n}$, and to $z$, there exists a $G$-invariant functional $x^* \in S_{X^*}$ such that 
	\begin{eqnarray*}
		x^* (z) &>& \sup \Big\{ x^* (x): x \in V_{n+1} \setminus B(p, \eps) \Big\} + \frac{r}{2^n} \\
		&=& \sup \Big\{ |x^* (x)|: x \in V_{n+1} \setminus (B(p, \eps) \cup B(-p, \eps)) \Big\} + \frac{r}{2^n} 
	\end{eqnarray*}
	since $V_{n+1} \setminus (B(p, \eps) \cup B(-p, \eps))$ is symmetric. Since $z \in V_n$, there exists $T \in \mathcal{L}_G(X, Y)$ such that 
	\begin{equation*}
		\|T(z)\| \geq \|T\|_B - \frac{\delta^2}{4^n}, \ \ \ \ \ \|S - T\| \leq \delta \left(1 - \frac{1}{2^n} \right), \ \ \ \ \mbox{and} \ \ \ \ S - T \ \mbox{is finite rank}.
	\end{equation*}
	So, 
	\begin{equation*}
		\delta \left(1 - \frac{1}{2^n} \right) \geq \|S - T\| \geq \|S - T\|_B \geq \|S\|_B - \|T\|_B = 1 - \|T\|_B 
	\end{equation*}	
	and then
	\begin{equation*}
		\|T\|_B \geq 1 - \delta \left(1 - \frac{1}{2^n} \right) > \frac{1}{2}. 
	\end{equation*}
	Also, 
	\begin{equation*}
		\|T\|_B \leq \|T - S\|_B + \|S\|_B \leq \|T-S\| + \|S\|_B \leq \delta\left(1 - \frac{1}{2^n}\right) + 1 < 2.
	\end{equation*}
	Thus, $\frac{1}{2} < \|T\|_B < 2$ and then
	\begin{equation*}
		\|T(z)\| \geq \|T\|_B - \frac{\delta^2}{4^n} > \frac{1}{2} - \frac{\delta^2}{4^n} > \frac{1}{4}.
	\end{equation*}
	Let $\widetilde{T} \in \mathcal{L}(X, Y)$ be given by 
	\begin{equation*}
		\widetilde{T}(x) := T(x) + \frac{1}{2^{n+1}} \delta x^* (x) T(z) \ \ \ (x \in X).
	\end{equation*}
	Then, it is clear that $\|\widetilde{T} - T\| \leq \frac{\delta}{2^{n+1}}$ and $\|S - \widetilde{T}\| \leq \delta(1 - \frac{1}{2^n})$. Since $\widetilde{T} - T$ is finite rank, we have that $S - \widetilde{T}$ is also finite rank. By assumption, we have that $\widetilde{T} \not\in A_{\eps}$. Thus, there is $x \in B \cap X_G$ such that $x \not\in B(p, \eps) \cup B(-p, \eps)$ and 
	\begin{equation*}
		\|\widetilde{T}(x)\| \geq \|\widetilde{T}\|_B - \frac{\delta^2}{4^{n+1}}.
	\end{equation*}
	It follows that $x \in V_{n+1} \setminus (B(p, \eps) \cup B(-p, \eps))$. However, arguing as in \cite[Lemma 4]{B}, we obtain that $|x^*(x)| \geq x^* (z) - \frac{\delta}{2^{n-5}}$, which is a contradiction with $x^* (z) > |x^* (x)| + r2^{-n}$. This proves the claim.

	\vspace{0.2cm}

	By \cite[Lemma 3]{B}, the set $\bigcap_{n=0}^{\infty} \overline{\bigcup_{j \geq n} \co(V_j)}$ is non-empty and not dentable, which contradicts our assumption.		
\end{proof}

Recall that an operator $T$ in $\mathcal{L} (X,Y)$ is an \emph{absolutely exposing operator for a subset $B$ of $X$} if there exists $x \in B$ such that every sequence $(x_n)$ in $B$ satisfying $ \lim_{n \rightarrow \infty} \|T(x_n)\| = \|T\|_B$ has a subsequence converging to $\theta x$ for some $\theta \in \mathbb{C}$ with $|\theta| =1$. 

\begin{theorem} \label{theorem-bourgain} Let $X$ be a real Banach space and $G \subseteq \mathcal{L}(X)$ a compact group. Let $B$ be a non-empty bounded, closed, absolutely convex, and $G$-invariant subset of $X$. Assume that every subset of $B$ is dentable. Then, for any Banach space $Y$, the set $A$ of absolutely strongly exposing operators $T \in \mathcal{L}_G(X, Y)$ for the set $B$ is a dense $G_{\delta}$-subset of $\mathcal{L}_G(X, Y)$. In fact, if $S \in \mathcal{L}_G(X, Y)$ and $\delta > 0$, then there is $T \in A$ such that $\|S - T\| \leq \delta$ and $S - T$ is compact. 
\end{theorem}

\begin{proof} Let us assume that $B$ is contained in $B_X$. For each $n \in \N$, consider the set $A_{1/n}$ as in Lemma \ref{lemma-bourgain}. We claim that $A_{1/n}$ is an open set. Let $T \in A_{1/n}$ and $\eta > 0$ be such that $S(T,\eta) \subseteq B(p, 1/n) \cup B(-p, 1/n)$ for some $p \in X$. 
If $U \in \mathcal{L}_G(X, Y)$ with $\|U - T\| < \frac{\eta}{3}$, then $S \left(U, \frac{\eta}{3} \right) \subseteq S(T, \eta)$; hence $U \in A_{1/n}$ and the claim follows. Since $A = \bigcap_{n=1}^{\infty} A_{1/n}$, it follows from Lemma \ref{lemma-bourgain} that $A$ is a dense $G_{\delta}$ subset of $\mathcal{L}_G(X, Y)$. The second statement follows from the arguments in the proof of \cite[Theorem 5]{B}.
\end{proof}

\begin{remark}
	By considering $\bigcup_{|t|=1} B(t \cdot p, \eps)$ instead of $B(p, \eps) \cup B(-p, \eps)$ and considering real parts of functionals, we can obtain complex versions of Lemma \ref{lemma-bourgain} and Theorem \ref{theorem-bourgain}.
\end{remark}

\noindent 
\textbf{Funding.} S. Dantas and J. Falcó were supported by grant PID2021-122126NB-C33 funded by MCIN/AEI/10.13039/501100011033 and by “ERDF A way of making Europe”. S. Dantas was also supported by the Spanish AEI Project PID2019 - 106529GB - I00 / AEI / 10.13039/501100011033. 
J. Falcó was also supported by MINECO and FEDER Project MTM2017-83262-C2-1-P. 
M. Jung was supported by NRF (NRF-2019R1A2C1003857), by POSTECH Basic Science Research Institute Grant (NRF-2021R1A6A1A10042944), and by a KIAS Individual Grant (MG086601) at Korea Institute for Advanced Study. 
\vspace{0.4cm}

\noindent
\textbf{Conflict of interest:} All authors declare that they have no conflict of interest. \\
\textbf{Data availability:} No data were used to support this study.

\vspace{0.4cm}
\noindent
\textbf{Acknowledgements:} 
The authors are grateful to Miguel Mart\'in for fruitful conversations on the topic of the paper. 
We also thank the anonymous referee for the carefully reading of the manuscript and for providing a number of suggestions which have improved its final form.

\end{document}